\documentclass[11pt]{amsart}
\usepackage{etex}

\usepackage{amsmath}
\usepackage{mathrsfs, euscript}
\usepackage{bbm}
\usepackage{amssymb}
\usepackage{amscd}
\usepackage[all,cmtip]{xy}
\usepackage{enumerate}

\usepackage{tikz}
\usetikzlibrary{calc,positioning}
\usepackage{graphicx}
\usepackage{color}
\usepackage{ifthen}
\definecolor{dgrey}{gray}{0 }
\definecolor{lgrey}{gray}{1}
\def\greyscale{1}
\ifthenelse{\greyscale>0}{
	\def\palette{{"dgrey", "lgrey"}}
	\definecolor{rline}{rgb}{0 0 0}
}

\newtheorem{theorem}{Theorem}[section]
\newtheorem*{theorem*}{Theorem}
\newtheorem{lemma}[theorem]{Lemma}
\newtheorem{example}[theorem]{Example}
\newtheorem{corollary}[theorem]{Corollary}
\newtheorem{proposition}[theorem]{Proposition}
\newtheorem{remark}[theorem]{Remark}
\newtheorem{definition}[theorem]{Definition}

\newcommand{\K}{\mathcal{K}}

\newcommand{\J}{\mathcal{J}}
\newcommand{\I}{\mathcal{I}}
\newcommand{\hh}{\mathcal{H}}
\newcommand{\LL}{\mathcal{L}}
\newcommand{\F}{\mathcal{F}}

\newcommand{\A}{\mathcal{A}}
\newcommand{\B}{\mathcal{B}}

\newcommand{\W}{\mathcal{W}}

\newcommand{\C}{\mathbb{C}}
\newcommand{\mrx}{\mathring{\mathcal{X}}}
\newcommand{\X}{\mathcal{X}}
\newcommand{\Y}{\mathcal{Y}}
\newcommand{\mry}{\mathring{\mathcal{Y}}}

\newcommand{\talpha}{\tilde{\alpha}}

\usepackage[all]{xy}

\begin{document}
\unitlength=1mm
\special{em:linewidth 0.4pt}
\linethickness{0.4pt}
\title[Free-Boolean]{ Free-free-Boolean independence for triples  of algebras}
\author{Weihua Liu}
\maketitle
\begin{abstract}
In this paper, we introduce the notion of free-free-Boolean independence relation for triples of algebras. We define free-free-Boolean cumulants ans show that the vanishing of mixed cumulants is equivalent to free-free-Boolean independence. A free-free -Boolean central limit law is studied.
 \end{abstract}
%\small
%\footnotesize
\section{Introduction}
In noncommutative probability,  independence relations between random variables  provide specific rules for calculations of  all their mixed moments. 
It was shown in \cite{Sp2} that  there are  exactly three commutative and associative independence relations, namely the classical independence,  Voiculescu's free independence \cite{VDN} and the Boolean independence \cite{SW}. 
Bi-free independence relation was introduced  by Voiculescu  as a  generalization of free independence relation to an independence relation for pairs of algebras. 
In bi-free probability,   left  and right regular representations of pairs of algebras on reduced free products  of vector spaces with specified vectors are studied  simultaneously \cite{Voi1}.  
Moreover, bi-free probability started a program of studying independence relations for pairs of algebras. 
For example,   conditionally bi-free independence, bi-Boolean independence, bi-monotone independence are developed in \cite{GS, GS1, GHS}.
 
In \cite{Liu3}, the author introduced a notion of mixed independence relations  which are defined via truncations of reduced free products of algebras. 
In this paper, we generalize this idea further, that is we study independence relations for triples of algebras.  
The specific independence relation we study in this paper is  free-free-Boolean independence.  This is the only commutative and associative independence relation defined  via  regular representations of algebras on reduced free products of vector spaces with specified vectors \cite{Liu3}.  
As in the combinatorial aspects of the other commutative independence relations, we introduce an associative  family of partitions which we call interval-bi-noncrossing partitions.  
We show that the families of interval-bi-noncrossing partitions are lattices.  
Then, we define free-free-Boolean cumulants via M\"obius inversion functions on interval-bi-noncrossing partitions and show that the vanishing of mixed free-free-Boolean cumulants is equivalent to free-free-Boolean independence.  
This allows us to obtain a central limit law for free-free-Boolean independence.

Besides this introduction section, the paper is organized as follows:  
In Section 2, we briefly review the constructions of mixed independence relations and define our free-free-Boolean independence relation. 
In Section 3,  we introduce a notion of  interval-bi-noncrossing partitions and study their lattice structures.
In Section 4,  we study  the M\"obius inversion functions on the lattices of interval-bi-noncrossing partitions.  Free-free-Boolean cumulants and combinatorially free-free-Boolean independence are introduced.
  In Section 5,  we show that the vanishing of mixed combinatorially free-free-Boolean cumulants is equivalent to the  free-free-Boolean independence.
  In Section 6, as an application of the main theorem in Section 5,  we  study a free-free-Boolean central limit law.
\section{Preliminaries and Notation}
In this section, we briefly review  notions and constructions of free independence and Boolean independence. The main purpose is to give a constructive definition for free-free Boolean independence. Let us start with some necessary definitions.
\begin{definition}\normalfont  Let $\I$ be an index set and  $(\A,\phi)$ be a noncommutative probability space where $\A$ is an algebra and $\phi$ is a linear functional on $\A$ such that $\phi(1_{\A})=1$. \\
 A family of unital subalgebras $\{\A_i|i\in I \}$ of $\A$ is said to be freely independent  if
$$\phi(x_1\cdots x_n)=0,$$
whenever $x_k\in\A_{i_k}$, $i_1\neq i_2\neq\cdots\neq i_n$ and $\phi(x_k)=0$ for all $k$. 

A family of (not necessarily unital) subalgebras $\{\A_i|i\in I \}$ of $\A$ is said to be Boolean independent if
$$\phi(x_1x_2\cdots x_n)= \phi(x_1)\phi(x_2)\cdots \phi(x_n),$$
whenever $x_k\in\A_{i_k}$ with $i_1\neq i_2\neq\cdots \neq i_n$. \\

A set of random variables $\{x_i\in\A|i\in \I\}$ is said to be freely(Boolean) independent if the family of unital(non-unital) subalgebras $\A_i$, which is generated by $x_i$ respectively, is freely (Boolean) independent.
\end{definition} 

\begin{definition}\normalfont  A \emph{vector space with a specified vector} is a  triple $(\X,\mrx, \xi)$ where $\X$ is a vector space, $\mrx$ is a codimension one subspace of $\mathcal{X}$  and $\xi\in\mathcal{X}\setminus\mrx$.
\end{definition}

Let $(\mathcal{X},\mrx, \xi)$ be a  vector space with a specified vector. 
Notice that $\X=\C\xi\oplus\mrx $,  there exists a unique linear functional $\phi$ on $\mathcal{X}$  such that $\phi(\xi)=1$  and $\ker(\phi)=\mrx$.  
We  denote by $\mathcal{L}(\mathcal{X})$ the algebra of linear operators on $\mathcal{X}$ and we define a  linear functional $\phi_\xi:\LL(\X)\rightarrow \C$ such that $\phi_\xi(T)=\phi(T\xi), T\in \LL(\X)$.  

Given a family of vector spaces with specified vectors  $(\X_i, \mrx_i,\xi_i)_{i\in \I}$ , their  reduced free product space $(\mathcal{X}, \mrx,\xi)=\underset{i\in \I}{*}(\X_i, \mrx_i,\xi_i)$ is given by $\X=\C\xi\oplus\mrx$ where
 $$\mrx=\bigoplus\limits_{n\geq 1}\left(\bigoplus\limits_{i_1\neq i_2\neq\cdots \neq i_n} \mathring{\mathcal{X}_{i_1}}\otimes\cdots \otimes\mrx_{i_n}\right). $$
For each $i\in \I$, we let
$$\X(\ell,i)= \mathbb{C}\xi\oplus\mrx(\ell,i)\quad \text{where}\quad \mrx(\ell,i)=\bigoplus\limits_{n\geq 1}\left(\bigoplus\limits_{i_1\neq i_2\neq\cdots \neq i_n, i_1\neq i} \mathring{\mathcal{X}_{i_1}}\otimes\cdots \otimes\mrx_{i_n}\right)$$
and
$$\X(r,i)=\mathbb{C}\xi\oplus\mrx(r,i)\quad \text{where}\quad \mrx(r,i)=\bigoplus\limits_{n\geq 1}\left(\bigoplus\limits_{i_1\neq i_2\neq\cdots \neq i_n, i_n\neq i} \mathring{\mathcal{X}_{i_1}}\otimes\cdots \otimes\mrx_{i_n}\right).$$
As was shown in  \cite{Voi1}, there are natural  linear isomorphisms: $V_i:\mathcal{X}_i\otimes \mathcal{X}(\ell,i) \rightarrow \mathcal{X}$ and $W_i: \mathcal{X}(r,i)\otimes\mathcal{X}_i \rightarrow \mathcal{X}$.
Therefore,  for each $i\in \I$, the algebra $\mathcal{L}(\mathcal{X}_i)$ has a left representation $\lambda_i$ and a right representation $\rho_i$, on $\X$, which are given by
$$\lambda_i(T)=V_i(T\otimes I_{\mathcal{X}(\ell,i)})V_i^{-1}$$
and 
$$\rho_i(T)=W_i(I_{\mathcal{X}(r,i)}\otimes T )W_i^{-1}$$ 
for every $T\in \mathcal{L}(\mathcal{X}_i)$, where $I_{\mathcal{X}(r,i)}$ and $I_{\mathcal{X}(\ell,i)}$ are the identity operators on ${\mathcal{X}(r,i)}$ and ${\mathcal{X}(\ell,i)}$ respectively.   

For each $i\in\I$, let $P_i$ be the projection from $\X$ onto the subspace  $\C\xi\oplus\mrx_i$ which vanishes on all the other direct summands. 
\begin{proposition}\label{simple Boolean}\normalfont
For any $a\in \LL(X_i)$, we have
$P_i\lambda_i(a)=\lambda_i(a)P_i$. 
\end{proposition}
\begin{proof}
Notice that 
$$
\X=\C\xi\oplus\mrx_i\oplus\X_i',
\quad \text{where}\quad \X_i= \bigoplus\limits_{j\neq i}\mrx_j
 \oplus \bigoplus_{n\geq 2}\left(\bigoplus_{\substack{i_1\neq i_2\neq \cdots \neq i_n}}
 \mrx_{i_1}\otimes\cdots\otimes\mrx_{i_n}\right).
$$
By direct computations, we have that  $$V_i^{-1}\X_i'=\X_i\otimes \mrx(\ell,i)\quad \text{and} \quad V_i^{-1}(\C\xi\oplus\mrx_i)=\X_i\otimes \xi$$ 
which are invariant under $\LL(\X_i)\otimes I_{\X(\ell,i)}$. Therefore, $\X_i'$ and $\C\xi\oplus\mrx_i$ are  invariant under $\lambda_i(a)$ for any $a\in\LL(\X_i)$. The statement  follows. 
%that $P_i\lambda_i(a)x=0$ for all $a\in\LL(\X_i)$, $x\in \X_i'$ and $P_i\lambda_i(a)x=\lambda_i(a)x$ for all $a\in\LL(\X_i)$, $x\in \C\xi\oplus\mrx_i$. 
\end{proof}
The same we have the following statement for $\rho_i$.
\begin{corollary}\normalfont
For any $a\in \LL(X_i)$, we have
$P_i\rho_i(a)=\rho_i(a)P_i$. 
\end{corollary}
 By Proposition 2.4 in \cite{Liu3}, $P_i\lambda_i(\cdot)P_{i}$ and $P_i\rho_i(\cdot)P_{i}$ are a same homomorphism from $\LL(\X_i)$ to $\LL(X)$.

\begin{definition} \normalfont
A triple of faces  in a noncommutative probability space $(\A,\phi)$ is an ordered triple $((B,\beta),$ $(C,\gamma), (D,\delta))$ where  $B,C,D$ are algebras and $\beta:B\rightarrow \A$, $\gamma:C\rightarrow \A$, $\delta:D\rightarrow \A$, are homomorphisms which are not necessarily unital.  If $B,C,D$ are subalgebras of $\A$  and $\beta,\gamma,\delta$ are inclusions, then the triple will be denoted by $(B,C,D)$.
\end{definition}

\begin{definition}\normalfont Let $\Gamma=\{((B_i,\beta_i),$ $(C_i,\gamma_i), (D_i,\delta_i))\}_{i\in \I}$ be a family of triples of faces in $(\A,\phi)$. 
The joint distribution of $\Gamma$ is the  functional $\mu_\Gamma: \underset{i\in \I}{\bigstar} (B_{i}\bigstar C_i\bigstar D_i)\rightarrow \C$ defined by $\mu_\Gamma=\phi\circ \alpha$, 
where $\underset{i\in \I}{\bigstar} (B_{i}\bigstar C_i \bigstar D_i)$ is the universal nonunital free product of $\{(B_i,C_i,D_i)\}_{i\in \I}$  and $\alpha:\underset{i\in \I}{\bigstar} (B_{i}\bigstar C_i \bigstar D_i)\rightarrow \A$ is the homomorphism such that $\alpha|_{B_i}=\beta_i$, $\alpha|_{C_i}=\gamma_i$ and $\alpha|_{D_i}=\delta_i$.
\end{definition}

\begin{definition}\normalfont  A three-faced  family of  random variables in a noncommutative probability space $(\A, \phi) $ is an ordered triple  $a=\{(b_i)_{i\in \I},(c_j)_{j\in \J},(d_k)_{k\in \K}\}$ of families of  random variables in  $(\A, \phi) $( i.e. the $b_i$, $c_j$ and $d_k$ are elements of $\A$).  The distribution $\mu_a$ of $a$ is the  functional 
$$\mu_a:\C\langle X_i, Y_j,Z_k|i\in \I, j\in \J,k\in \K\rangle \rightarrow \C$$ 
such that $\mu_a=\phi\circ \alpha $ where $\alpha:\C\langle X_i, Y_j, Z_k|i\in \I, j\in \J, k\in \K\rangle \rightarrow \A$ is the homomorphism  such that $\alpha(X_i)=b_i$, $\alpha(Y_j)=c_j$ and $\alpha(Z_k)=d_k$.
\end{definition}

In the following context, for convenience, we assume that $B_i,C_i,D_i$ are subalgebras of $\A$.

\begin{definition}\label{bifree-boolean} \normalfont 
Let $\Gamma=\{(B_i,C_i,D_i)\}_{i\in \I}$ be a family of  triples of faces in $(\A,\phi)$. 
%For each $i\in \I$, suppose that  $B_i,C_i,D_i$ are (not necessarily unital) subalgebras of $\A$ and $D_i $ is not necessarily unital.   
Suppose that there is a family of vector spaces with specified vectors $ (\mathcal{X}_i, \mrx_i,\xi_i)_{i\in \I}$ ,  (not necessarily) homomorphisms $\ell_i:B_i\rightarrow \mathcal{L}(\X_i)$, $r_i:C_i\rightarrow \LL(\X_i)$  and $m_i:D_i\rightarrow \LL(\X_i)$.  
Let  $( \mathcal{X}, \mrx,\xi)$ be the reduced free product of $ (\mathcal{X}_i, \mrx_i,\xi_i)_{i\in \I}$ and $\phi_\xi$ is the functional associated with $\xi$ on $\LL(\X)$.  
We say that the family of triples of faces  $\{(B_i,C_i,D_i)\}_{i\in \I}$ is free-free-Boolean independent if  the joint distribution of $\{(B_i,\lambda_i(\ell_i(\cdot))),(C_i,\rho_i(r_i(\cdot))),(D_i,P_i\lambda_i(m_i(\cdot))P_i)\}_{i\in \I}$, which is in $(\LL(X),\phi_{\xi})$, is equal to the joint distribution of $\{(B_i,C_i,D_i)\}_{i\in \I}$. 
In this case, we  say that the family $\{(B_i,C_i),i\in \I)\}_{i\in \I}$ is bifree independent and the family $\{(B_i,D_i),i\in \I)\}_{i\in \I}$ is free-Boolean independent.
\end{definition}

\begin{remark}\label{Largest triple} \normalfont Notice that $\lambda_i$, $\rho_i$ and $P_i\lambda_iP_i$ are injective. Therefore,  they have left inverse, that is there exists $\ell_i$, $r_i$ and $m_i$ such that $\ell_i(\lambda_i(\cdot))=I_{\LL(\X_i)}$, $r_i(\rho_i(\cdot))=I_{\LL(\X_i)}$ and $m_i(P_i\lambda_i(\cdot)P_i)=I_{\LL(\X_i)}$. Thus  $\{(\lambda_i(\LL(\X_i)), \rho_i(\LL(\X_i)),P_i\lambda_i(\LL(\X_i))P_i )\}$  is a family of free-free-Boolean triples of faces in $(\LL(\X),\phi_\xi)$.
\end{remark}

\begin{proposition}\label{subagbras of ffb triples}\normalfont Let $\Gamma=\{(B_i,C_i,D_i)\}_{i\in \I}$ be a family of free-free-Boolean triples of faces in $(\A,\phi)$ and $\Gamma'=\{(B'_i,C'_i,D'_i)\}_{i\in \I}$ be a family of  triples of faces such that $B'_i\subset B_i$, $C'_i\subset C_i$  and $D'_i\subset D_i$, $i\in \I$. Then $\Gamma'=\{(B'_i,C'_i,D'_i)\}_{i\in \I}$ is  free-free-Boolean  in $(\A,\phi)$.
\end{proposition}
\begin{proof}
Since $B'_i\subset B_i$, $C'_i\subset C_i$  and $D'_i\subset D_i$, the homomorphisms $\lambda_i(\ell_i(\cdot)),\,\rho_i(r_i(\cdot)),\,P_i\lambda_i(m_i(\cdot))P_i$ in the preceding definition are well defined on $B_i'$, $C_i'$, $D_i'$ for all $i\in \I$. Therefore, $\Gamma'=\{(B'_i,C'_i,D'_i)\}_{i\in \I}$ is  free-free-Boolean  in $(\A,\phi)$.
\end{proof}

\begin{remark}\normalfont
 Given a probability space $(\A, \phi)$, it is well known that  Boolean independence relation is defined for non-unital algebras in the sense that two Boolean independent  algebras $\B_1$ and $ \B_2$ do not contain the unit of $\A$ or else the expectation is a homomorphism from  $B_1\bigvee\B_2$ to  $\C$, where $B_1\bigvee\B_2$ is the algebra generated by $\B_1,\B_2$.  
Therefore, in the Definition \ref{bifree-boolean}, we do not require the homomorphisms $\lambda_i$ and $\rho_i$ to be unital so that it allows  each triple $(B_i,C_i,D_i)$ of faces to be chose arbitrarily. For instance, in our definitions, $B_i$ and $C_i$ can be Boolean independent,	 monotone independent, etc.  
\end{remark}

By Definition 1.10, Lemma 1.11-1.12 in \cite{Voi1},  it is a routine to show that our free-free-independence relation is independent of choices of representations $\ell_i, r_i$ and $m_i$.  

\begin{proposition}\normalfont Let $\{(B_i,C_i,D_i)\}_{i\in \I}$ be a family of free-free-Boolean independent triples of faces in $(\A,\phi)$, where $\I$ is an index set. Let $\underset{k\in \K}\amalg \I_k$ be a partition of $\I$.  For each $k\in \K$,  let $B_k$ be the unital algebra generated by $\{B_k|i\in \I_k\}$, $C_k$ be the unital algebra generated by $\{C_k|i\in \I_k\}$, $D_k$ be the non-unital algebra generated by $\{D_k|i\in \I_k\}$. Then,  $\{(B_k,C_k,D_k)\}_{i\in \I}$ is a family of free-free-Boolean independent triples of faces in $(\A,\phi)$.
\end{proposition}
\begin{proof}
By Remark \ref{Largest triple} and Proposition \ref{subagbras of ffb triples}, it is sufficient prove the statement under the assumption that $\{(B_i,C_i,D_i)\}_{i\in \I}$ and $(\A,\phi)$ are $\{(\lambda_i(\LL(\X_i)), \rho_i(\LL(\X_i)),P_i\lambda_i(\LL(\X_i))P_i )\}$ and $(\LL(\X),\phi)$ respectively, where $ (\mathcal{X}_i, \mrx_i,\xi_i)_{i\in \I}$ is a family of vector spaces with specified vectors   $( \mathcal{X}, \mrx,\xi)$ is their reduced free product and $\phi$ is  linear functional  associated with $\xi$ on $\LL(\X)$. 
For each $k\in \K$, let $(\Y_k,\mry_k,\xi_k)$  be the reduced free product of $ (\mathcal{X}_i, \mrx_i,\xi_i)_{i\in \I_k}$. 
Then, by Remark 1.13 in \cite{Voi1},  $( \mathcal{X}, \mrx,\xi)$ is also the reduced free product of   $ (\Y_k, \mry_k,\xi_k)_{k\in \K}$.
 Let  $\lambda'_k$ be the left regular representation of $\LL(Y_k)$ on $\X$, $\rho'_k$ be the right regular representation of $\LL(Y_k)$ on $\X$ and $P'_k$ be the projection from $\X$ onto $\C\xi\oplus \mry_k$.
For each $i\in \I_k$, we have $\lambda_i(\LL(\X_i))\subseteq \lambda'_k(\LL(Y_k))$, $\rho_i(\LL(\X_i))\subseteq \rho'_k(\LL(Y_k))$ and $P_i\lambda_i(\LL(\X_i))P_i\subseteq P'_k\lambda'_k(\LL(Y_k))P_k'$.
Notice that  $\{(\lambda_k'(\LL(\Y_k)), \rho_k'(\LL(\Y_k)),P_k'\lambda_k'(\LL(\Y_k))P_k')\}_{k\in \K}$ is a family of free-free-Boolean independent triples of faces in $(\LL(\X),\phi)$, by Proposition \ref{subagbras of ffb triples}, the proof is finished.

\end{proof}

\begin{definition}\normalfont 
Let $B, C$ be two  subalgebras of a probability space $(A,\phi)$. We say that $B$ is \emph{monotone} to $C$ if  $$\phi(x_1\cdots x_{k-1}x_kx_{k+1}\cdots x_n)=\phi(x_k)\phi(x_1\cdots x_{k-1}x_{k+1}\cdots x_n)$$
whenever $x_{k-1},x_{k+1}\in C$ and $x_k\in B$.
\end{definition}

Given a family of algebras $\{A_i\}_{i\in \I}$, we denote by $\underset{i\in\I}\bigvee A_i$ the nonunital algebra generated by $\{A_i\}_{i\in \I}$. By Proposition 2.16 in \cite{Voi1} and Proposition 3.8 in \cite{Liu3}, we have the following result.

\begin{proposition}\normalfont 
Let $\{(B_i,C_i,D_i)\}_{i\in \I}$ be a family of free-free-Boolean independent triples of faces in $(\A,\phi)$, where $\I$ is an index set and  $L\subset \I$. Then
\begin{itemize}
\item the subalgebra $\underset{i\in L}{\bigvee} B_i$ is monotone to  $\underset{i\in I\setminus L}{\bigvee} D_i$ in $(\A,\phi)$,
\item the subalgebra $\underset{i\in L}{\bigvee} C_i$ is monotone to  $\underset{i\in I\setminus L}{\bigvee} D_i$ in $(\A,\phi)$,
\item the subalgebra $\underset{i\in L}{\bigvee} B_i$ and   $\underset{i\in I\setminus L}{\bigvee} C_i$ are classically independent in $(\A,\phi)$.
\end{itemize} 
\end{proposition}

\begin{definition} \normalfont  Let $a=\{(b_i)_{i\in \I},(c_j)_{j\in \J}, (d_k)_{k\in\K}\}$ and $a'=\{(b'_i)_{i\in \I},(c'_j)_{j\in \J}, (d'_k)_{k\in\K}\}$ be a pair of  three faced families of  random variables in a probability space  $(\A, \phi) $. 
We say that $a$ and $a'$ are \emph{free-free-Boolean independent} if $(B,C,D)$ and $(B',C',D')$ are free-free-Boolean independent, where $B,B',C,C',$ and $D'$ are nonunital algebras generated by $(b_i)_{i\in \I}$, $(b'_i)_{i\in \I}$, $(c_j)_{j\in \J}$, $(c'_j)_{j\in \J}$, $(d_k)_{k\in \K}$ and $(d'_k)_{k\in \K}$ respectively. 
If $a=\{(b_i)_{i\in \I},(c_j)_{j\in \J}, (d_k)_{k\in\K}\}$ and $a'=\{(b'_i)_{i\in \I},(c'_j)_{j\in \J}, (d'_k)_{k\in\K}\}$ are free-free-Boolean independent, then the joint distribution of $\{(b_i+b'_i)_{i\in \I},(c_j+c'_j)_{j\in \J},(d_k+d'_k)_{k\in\K}\}$ is determined. This defines an additive free-free-Boolean convolution $\boxplus\boxplus\uplus$ on distributions of three-faced families of random variables with a triple of index sets $(\I,\J,\K)$
$$\mu_{\{(b_i+b'_i)_{i\in \I},(c_j+c'_j)_{j\in \J},(d_k+d'_k)_{k\in\K}\}}=\mu_{\{(b_i)_{i\in \I},(c_j)_{j\in \J}, (d_k)_{k\in\K}\}}\boxplus\boxplus\uplus \mu_{\{(b'_i)_{i\in \I},(c'_j)_{j\in \J}, (d'_k)_{k\in\K}\}}.$$
Similarly,  we can define   multiplicative,  additive-additive-multiplicative, multiplicative-additive-additive free-Boolean convolutions, etc.
\end{definition}

\section{Interval Bi-noncrossing partitions}

In this section, we introduce  combinatorial tools for characterizing free-free-Boolean triples of faces. In the following context, we  denote by $[n]$ the set $\{1,...,n\}$.    

\begin{definition}\label{partition}\normalfont Let $S$ be a totally ordered set:
\begin{itemize}
\item[1.] A \emph{partition} $\pi$ of a set $S$ is a collection of disjoint, nonempty sets $V_1,...,V_r$ whose union is $S$.  $V_1,...,V_r$ are called \emph{blocks} of $\pi$. The collection of all partitions of $S$ will be denoted by $P(S)$.
\item[2.] Given two partitions $\pi$ and $\sigma$, we say $\pi\leq \sigma$ if each block of $\pi$ is contained in a block of $\sigma$. This relation  is called the \emph{reversed refinement order}.
\item[3.]A partition $\pi\in P(S)$ is a \emph{noncrossing} partition if there is no quadruple $(s_1,s_2,r_1,r_2)$ such that $s_1<r_1<s_2<r_2$, $s_1,s_2\in V$, $r_1,r_2\in W$ and $V,W$ are two different blocks of $\pi$.  
\item[4.]A partition $\pi\in P(S)$ is an \emph{interval} partition if there is no triple $(s_1,s_2,r)$ such that $s_1<r<s_2$, $s_1,s_2\in V$, $r\in W$ and $V,W$ are two different blocks of $\pi$.  
\item[5.] A block $V$ of a partition  $\pi\in P(S)$ is said to be an \emph{inner} block if there is block $W\in \pi$ and  $s,t\in W$ such that $s<v<t$ for all $v\in V$. A block is an \emph{exterior} block if it is not inner.
\item[6.] Let $\omega:[k]\rightarrow I$.  We denote by ker $\omega$ the element of $P([k])$ whose blocks are sets $\omega^{-1}(i)$, $i\in I$.
\end{itemize}
\end{definition}

\begin{definition} \normalfont 
Given $\chi:[n] \rightarrow \{\ell,c,r\}$ a map from $[n]$ to  the set of letters $\{\ell, c, r\}$ with
$$\chi^{-1}\{\ell, c\}=\{i_1<i_2<\cdots< i_p\}\,\,\,\,\,\text{and}\,\,\,\,\, \chi^{-1}\{r\}=\{i_{p+1}>i_{p+2}>\cdots>i_n\}, $$
$\prec_{\chi}$ is a \emph{total order} on $[n]$ defined by
$$i_1\prec_{\chi}i_2\prec_{\chi}\cdots\prec_{\chi}i_p\prec_{\chi}i_{p+1}\prec_{\chi}i_{p+2}\prec_{\chi}\cdots<i_n.$$
Let $\pi\in P(n)$. $\pi$ is said to be a \emph{$\chi$-noncrossing} partition if $\pi$  is a  noncrossing partition with respect to the order $\prec_{\chi}$. We denote by $NC(\chi)$ the set of noncrossing partitions of $[n]$ with respect to the order $\prec_{\chi}$. $\pi$ is said to be  a \emph{$\chi$-interval} partition if $i,j,k$ are in the same block whenever   $i<j<k$, $i\sim k$, $\chi(j)=c$.
$\pi$ is said to be an \emph{interval-bi-noncrossing} partition with respect to $\chi$ if $\pi$ is  $\chi$-noncrossing  and  $\chi$-interval.  We denote by $IBNC(\chi)$ the set of all interval-bi-noncrossing partitions with respect to $\chi$.
\end{definition}

Recall that the family of bi-noncrossing partitions in \cite{CNS} is defined as follows.

\begin{definition} \normalfont 
Given $\bar{\chi}:[n] \rightarrow \{\ell,r\}$ a map from $[n]$ to  the set of letters $\{\ell, c, r\}$ with
$$\bar\chi^{-1}\{\ell\}=\{i_1<i_2<\cdots< i_p\}\,\,\,\,\,\text{and}\,\,\,\,\, \bar\chi^{-1}\{r\}=\{i_{p+1}>i_{p+2}>\cdots>i_n\}, $$
$\prec_{\bar\chi}$ is a \emph{total order} on $[n]$ defined by
$$i_1\prec_{\chi}i_2\prec_{\chi}\cdots\prec_{\chi}i_p\prec_{\chi}i_{p+1}\prec_{\chi}i_{p+2}\prec_{\chi}\cdots<i_n.$$
A partition $\pi$ is said to be a \emph{bi-noncrossing} partition with respect $\bar\chi$ if $\pi$  is a  noncrossing partition with respect to the order $\prec_{\chi}$. We denote by $BNC(\bar\chi)$ the set of bi-noncrossing partitions with respect to $\bar\chi$.
\end{definition}

\begin{remark}\normalfont  If $\chi^{-1}(c)=\emptyset$, then $NC(\chi)$ is the set of bi-noncrossing partitions $BNC(\chi)$.
\end{remark}

\begin{lemma}\normalfont
Let $\chi:[n] \rightarrow \{\ell,c,r\}$. If $\chi^{-1}(c)\subset\{1,n\}$, then $IBNC(\chi)=NC(\chi)$.
\end{lemma}
\begin{proof}
Since there is no $i,j\in[n]$ such that either $i<1<j$ or $i<n<j$, all partitions are $\chi$-interval partitions. Therefore, all $\chi$-noncrossing partitions are interval-bi-noncrossing partitions with respect to $\chi$. The proof is complete. 
\end{proof}

Given a  $\chi$, we can associate it a diagram as follows.  For each  $k=1,\cdots,n$, we place a node labeled $k$ at the position $(-1,n-k)$ if $k\in \chi^{-1}\{\ell,c\}$ and the position $(1,n-k)$ if $k\in \chi^{-1}\{r\}$.   We use white balls to denote nodes $k$  if $\chi(k)=c$ and draw a horizontal dashed lines through white balls.

\begin{example}\normalfont  Let $n=8$,  $\chi^{-1}(c)=\{4,6\}$, $\chi^{-1}(\ell)=\{1,2,8\}$ and $\chi^{-1}(r)=\{3,5,6\}$. The diagram associated with $\chi$ is the following.
$$\begin{tikzpicture}[baseline]
		\draw[thick, dashed] (-1,0.5) -- (-1, -4) -- (1,-4) -- (1,0.5);
      \draw[thick, dashed]  (-1, -1.5) -- (1,-1.5) ;
       \draw[thick, dashed] (-1, -2.5) -- (1,-2.5) ;
		\def\colours{{0,0,0,1,0,1,0,0,0}}
		\def\sidez{{-1,-1,1,-1,1,-1,1,-1}}
		\foreach \y in {0,...,7} {
			\pgfmathtruncatemacro{\nodename}{\y+1}
			\pgfmathtruncatemacro{\sd}{\sidez[\y]}
			\pgfmathparse{\palette[\colours[\y]]}
			\def\clr{\pgfmathresult}
			\node (ball\nodename) [draw, shade, circle, ball color=\clr, inner sep=0.08cm] at (\sd, -\y*0.5) {};
			\ifthenelse{\sd=1}{\node[right] at (\sd, -\y*0.5) {\nodename}}
					{\node[left] at (\sd, -\y*0.5) {\nodename}};
		}
		\node[right] at (-1.1,-4.5) {Diagram of $\chi$};
	\end{tikzpicture}
	$$
\end{example}

Roughly speaking, a partition $\pi$ is a $\chi-$noncrossing  partition if one can connect blocks of $\pi$ noncrossingly inside the above diagram. A partition $\pi$ is a $\chi-$interval if  each block of $\pi$ which goes across a dashed horizontal line contains the node on it.

\begin{example}\normalfont
Let $\chi$ be as in the preceding example. Given partitions $\pi_1=\{\{1,8\},\{2,3,4,5,6,7\}\}$, $\pi_2=\{\{1,2\},\{3,5,7,8\},\{4,6\}\}$ and  $\pi_3=\{\{1,2\},\{3,4,6,8\},\{5\},\{7\}\}$, they can be illustrated as follows. 
$$
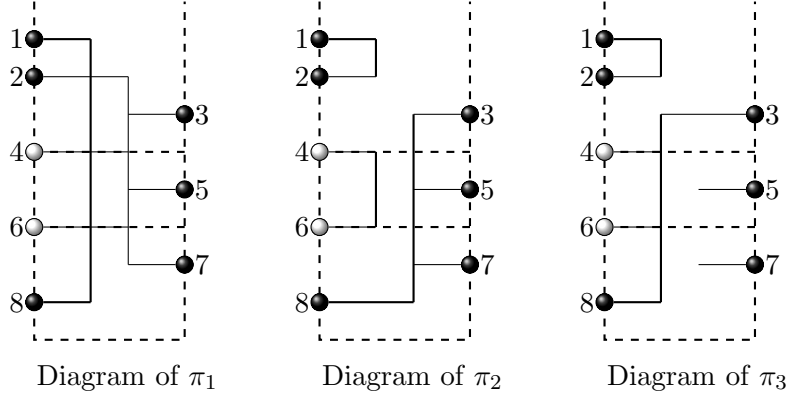
\begin{tikzpicture}[baseline]
		\draw[thick, dashed] (-1,0.5) -- (-1, -4) -- (1,-4) -- (1,0.5);
      \draw[thick, dashed]  (-1, -1.5) -- (1,-1.5) ;
       \draw[thick, dashed] (-1, -2.5) -- (1,-2.5) ;
		\def\colours{{0,0,0,1,0,1,0,0,0}}
		\def\sidez{{-1,-1,1,-1,1,-1,1,-1}}
		\foreach \y in {0,...,7} {
			\pgfmathtruncatemacro{\nodename}{\y+1}
			\pgfmathtruncatemacro{\sd}{\sidez[\y]}
			\pgfmathparse{\palette[\colours[\y]]}
			\def\clr{\pgfmathresult}
			\node (ball\nodename) [draw, shade, circle, ball color=\clr, inner sep=0.08cm] at (\sd, -\y*0.5) {};
			\ifthenelse{\sd=1}{\node[right] at (\sd, -\y*0.5) {\nodename}}
					{\node[left] at (\sd, -\y*0.5) {\nodename}};
		}
		\draw[rline,thin] (ball2)--(0.25,-0.5);
		\draw[rline,thin] (ball3)--(0.25,-1);
		\draw[rline,thin] (ball4)--(0.25,-1.5);
		\draw[rline,thin] (ball5)--(0.25,-2);
		\draw[rline,thin] (ball6)--(0.25,-2.5);
		\draw[rline,thin] (ball7)--(0.25,-3);
		\draw[rline, thin] (0.25,-0.5) -- (0.25,-3);
		
	    \draw[rline, thick] (ball1)--(-0.25,-0);
		\draw[rline,thick] (ball8)--(-0.25,-3.5);
		\draw[rline, thick] (-0.25,-0) -- (-0.25,-3.5);

		\node[right] at (-1.1,-4.5) {Diagram of $\pi_1$};
	\end{tikzpicture}
\quad\quad	
\begin{tikzpicture}[baseline]
		\draw[thick, dashed] (-1,0.5) -- (-1, -4) -- (1,-4) -- (1,0.5);
      \draw[thick, dashed]  (-1, -1.5) -- (1,-1.5) ;
       \draw[thick, dashed] (-1, -2.5) -- (1,-2.5) ;
		\def\colours{{0,0,0,1,0,1,0,0,0}}
		\def\sidez{{-1,-1,1,-1,1,-1,1,-1}}
		\foreach \y in {0,...,7} {
			\pgfmathtruncatemacro{\nodename}{\y+1}
			\pgfmathtruncatemacro{\sd}{\sidez[\y]}
			\pgfmathparse{\palette[\colours[\y]]}
			\def\clr{\pgfmathresult}
			\node (ball\nodename) [draw, shade, circle, ball color=\clr, inner sep=0.08cm] at (\sd, -\y*0.5) {};
			\ifthenelse{\sd=1}{\node[right] at (\sd, -\y*0.5) {\nodename}}
					{\node[left] at (\sd, -\y*0.5) {\nodename}};
		}
		\draw[rline,thin] (ball2)--(-0.25,-0.5);
		\draw[rline,thin] (ball3)--(0.25,-1);
		\draw[rline,thin] (ball4)--(-0.25,-1.5);
		\draw[rline,thin] (ball5)--(0.25,-2);
		\draw[rline,thin] (ball6)--(-0.25,-2.5);
		\draw[rline,thin] (ball7)--(0.25,-3);

	    \draw[rline, thick] (ball1)--(-0.25,-0);
		\draw[rline,thick] (ball8)--(0.25,-3.5);

		\draw[rline, thick] (0.25,-1) -- (0.25,-3.5); 
		\draw[rline, thick] (-0.25,-1.5) -- (-0.25,-2.5);
		\draw[rline, thick] (-0.25,-0) -- (-0.25,-0.5);

		\node[right] at (-1.1,-4.5) {Diagram of $\pi_2$};
	\end{tikzpicture}	
\quad\quad		
	\begin{tikzpicture}[baseline]
		\draw[thick, dashed] (-1,0.5) -- (-1, -4) -- (1,-4) -- (1,0.5);
      \draw[thick, dashed]  (-1, -1.5) -- (1,-1.5) ;
       \draw[thick, dashed] (-1, -2.5) -- (1,-2.5) ;
		\def\colours{{0,0,0,1,0,1,0,0,0}}
		\def\sidez{{-1,-1,1,-1,1,-1,1,-1}}
		\foreach \y in {0,...,7} {
			\pgfmathtruncatemacro{\nodename}{\y+1}
			\pgfmathtruncatemacro{\sd}{\sidez[\y]}
			\pgfmathparse{\palette[\colours[\y]]}
			\def\clr{\pgfmathresult}
			\node (ball\nodename) [draw, shade, circle, ball color=\clr, inner sep=0.08cm] at (\sd, -\y*0.5) {};
			\ifthenelse{\sd=1}{\node[right] at (\sd, -\y*0.5) {\nodename}}
					{\node[left] at (\sd, -\y*0.5) {\nodename}};
		}
		\draw[rline,thin] (ball2)--(-0.25,-0.5);
		\draw[rline,thin] (ball3)--(-0.25,-1);
		\draw[rline,thin] (ball4)--(-0.25,-1.5);
		\draw[rline,thin] (ball5)--(0.25,-2);
		\draw[rline,thin] (ball6)--(-0.25,-2.5);
		\draw[rline,thin] (ball7)--(0.25,-3);
				
	    \draw[rline, thick] (ball1)--(-0.25,-0);
		\draw[rline,thick] (ball8)--(-0.25,-3.5);
		
		\draw[rline, thick] (-0.25,-1) -- (-0.25,-3.5);
		\draw[rline, thick] (-0.25,-0) -- (-0.25,-0.5);
		
		\node[right] at (-1.1,-4.5) {Diagram of $\pi_3$};
	\end{tikzpicture}
$$

$\pi_1$ is a noncrossing partition but it is neither a $\chi$-interval partition nor a $\chi$-noncrossing partition. $\pi_2$ is  a  $\chi$-noncrossing partition but not a $\chi$-interval partition. $\pi_3$ is  an interval-bi-noncrossing partition with respect to $\chi$. 
\end{example}

Now, we turn to study  relations between $IBNC(\chi)$ and $NC(\chi)$.
%Actually, the relations between $IBNC(\chi)$ and $ NC(\chi)$ is very similar to the relations between 
%{\bf Explain: }Since the values of $\chi$ at $1$ and $n$ do not change $IBNC(\chi)$, in the rest of this section, we will assume that $1,n\in \chi^{-1}\{c\}$.\\
We denote by $[n_1,n_2]$ the set $\{n_1,n_1+1,\cdots,n_2\}$.  
  Suppose that    $\chi^{-1}\{c\}\cap [2,n-1]=\{l_1<\cdots< l_m\}$,  let $l_0=1$ and $l_{m+1}=n$. We  define the following maps :
\begin{itemize}
\item For each $i=1,\cdots,m+1$, let $\alpha_i: IBNC(\chi)\rightarrow P([l_{i-1},l_{i}])$ such that $\alpha_i(\pi)$ is the restriction of $\pi$ to $[l_{i-1},l_{i}]$ and  $\chi_i$ be the restriction of $\chi$ to $[l_{i-1},l_{i}]$.
\item Let $\alpha': IBNC(\chi)\rightarrow P([l_1,n])$ such that $\alpha'(\pi)$ is the restriction of $\pi$ to the  set $[l_1,n]$ and let $\chi'$ be the restriction of $\chi$ to the set $[l_1,...,n]$.
\end{itemize}

Since restrictions of partitions to do not turn any exterior block into inner, the restrictions $\alpha_1(\pi)$ and $\alpha'(\pi)$ of  a $\chi$-interval partition $\pi$ are $\chi_1$-interval and $\chi'$-interval respectively. On the other hand, restrictions of partitions  do not change the order $\prec_\chi$ , thus the restrictions $\alpha_1(\pi)$ and $\alpha'(\pi)$ of  a $\chi$-noncrossing partition $\pi$ are $\chi_1$-noncrossing and $\chi'$-noncrossing respectively. Therefore, the range of $\alpha_1$ is contained in  $IBNC(\chi_1)$ and the range of $\alpha'$ is contained in  $IBNC(\chi')$.  Notice that $\chi_1$ can be $c$ only at $1$ and $l_1$, it follows that  $IBNC(\chi_1)=NC(\chi_1).$

\begin{lemma}\label{injective}
 Let  $\alpha_1':IBNC(\chi)\rightarrow IBNC(\chi_1)\times IBNC(\chi')$ such that 
$$\alpha'_1(\pi)=(\alpha_1(\pi), \alpha'(\pi)). $$ Then $\alpha_1'$ is an injective map.
\end{lemma}
\begin{proof}

Let $\pi_1, \pi_2\in IBNC(\chi)$ such that $\pi_1\neq \pi_2$. Then, there exists a block $V\in \pi_1$ such that $V\not\in \pi_2$ and there exists a block $W\in \pi_2$ such that $V\cap W\neq \emptyset$.  We assume that $a\in V\cap W$, then  we have the following  cases:

1. If $l_1\in V$ and $l_1\not\in W$,  then $V\cap[1,l_1]\neq W\cap[1,l_1]$. Thus $\alpha_1(\pi_1)\neq \alpha_1(\pi_2)$ which implies that $\alpha_1'(\pi_1)\neq \alpha_1'(\pi_1)$.

2 If $l_1\in V$ and $l_1\in W$, then $V\cap[1,l_1]\neq W\cap[1,l_1]$ or $V\cap[l_1,n]\neq W\cap[l_1,n]$. Therefore, $\alpha_1(\pi_1)\neq \alpha_1(\pi_2)$ or $\alpha'(\pi_1)\neq \alpha'(\pi_2)$. 

3. Suppose that  $l_1$ is not contained in $V$ and $W$.  If $a>l_1$, then $V\subset [l_1,n]$ and $W\subset [l_1,n]$ since $\pi_1$ and $\pi_2$ are $\chi$-interval partitions. In this case, $V$ is a block of $\alpha'(\pi_1)$ and $W$ is a block of $\alpha'(\pi_2)$, thus $\alpha'(\pi_1)\neq \alpha'(\pi_2)$. Similarly, if $a<l_1$, then $\alpha_1(\pi_1)\neq \alpha_1(\pi_2)$.

The proof is complete.
\end{proof}

In the following context, we  assume that 
$$\chi^{-1}\{\ell, c\}=\{i_1<i_2<\cdots< i_p\}\,\,\,\,\,\text{and}\,\,\,\,\, \chi^{-1}\{r\}=\{i_{p+1}>i_{p+2}>\cdots>i_n\}, $$ 
$$\chi^{-1}\{\ell, c\}\cap [1,l_1]=\{i_1<i_2<\cdots< i_s=l_1\}\,\,\,\,\,\text{and}\,\,\,\,\, \chi^{-1}\{r\}\cap [1,l_1]=\{i_{t}>i_{t+1}>\cdots>i_n\},$$ 
$$\chi^{-1}\{\ell, c\}\cap [l_1,n]=\{l_1=i_s<i_2<\cdots< i_p\}\,\,\,\,\,\text{and}\,\,\,\,\, \chi^{-1}\{r\}\cap [l_1,n]=\{i_{p+1}>i_{p+2}>\cdots>i_{t-1}\}.$$ 
The diagrams of $\chi$, $\chi_1$ and $\chi'$ can be simply illustrated as follows.
$$\begin{tikzpicture}[baseline]
		\draw[thick, dashed] (-1,0) -- (-1, -3) -- (1,-3) -- (1,0);
      \draw[thick, dashed]  (-2, -1) -- (2,-1) ;
   \node [draw, shade, circle, ball color=white, inner sep=0.08cm] at (-1, -1) {};  
    \node[left] at (-1,-1) {$l_1$};
         \node[left] at (-1,-1.5) {$i_s$};
      \node[left] at (-1,-2) {$\vdots$};
         \node[left] at (-1,-2.5) {$i_p$};
   \node[right] at (1,-2.5) {$i_{p+1}$};
        \node[right] at (1,-2) {$\vdots$};
           \node[right] at (1,-1.5) {$i_{t-1}$};
           \node[right] at (1,-0.5) {$i_{t}$};		
			\node[right] at (-1.1,-4.5) {Diagram of $\chi$};
	\end{tikzpicture}
\quad\quad
	\begin{tikzpicture}[baseline]
		\draw[thick, dashed] (-1,0) -- (-1, -1.5) -- (1,-1.5) -- (1,0);
      \draw[thick, dashed]  (-2, -1) -- (2,-1) ;
   \node [draw, shade, circle, ball color=white, inner sep=0.08cm] at (-1, -1) {};  
    \node[left] at (-1,-1) {$l_1$};
         \node[left] at (-1,-0.5) {$i_{s-1}$};	
           \node[right] at (1,-0.5) {$i_{t}$};		
			\node[right] at (-1.1,-4.5) {Diagram of $\chi_1$};
	\end{tikzpicture}
	\quad\quad
	\begin{tikzpicture}[baseline]
		\draw[thick, dashed] (-1,0) -- (-1, -3) -- (1,-3) -- (1,0);
   
   \node [draw, shade, circle, ball color=white, inner sep=0.08cm] at (-1, -1) {};  
    \node[left] at (-1,-1) {$l_1$};
         \node[left] at (-1,-1.5) {$i_s$};
      \node[left] at (-1,-2) {$\vdots$};
         \node[left] at (-1,-2.5) {$i_p$};
   \node[right] at (1,-2.5) {$i_{p+1}$};
        \node[right] at (1,-2) {$\vdots$};
           \node[right] at (1,-1.5) {$i_{t-1}$};
   
			\node[right] at (-1.1,-4.5) {Diagram of $\chi'$};
	\end{tikzpicture}
$$

\begin{lemma}\label{Surjective}
 Let  $\alpha_1':IBNC(\chi)\rightarrow IBNC(\chi_1)\times IBNC(\chi')$ such that 
$$\alpha'_1(\pi)=(\alpha_1(\pi), \alpha'(\pi)). $$Then $\alpha_1'$ is a surjective map.
\end{lemma}
\begin{proof}
Let $\pi_1\in IBNC(\chi)$ and $\pi'\in IBNC(\chi')$. 
Suppose that $\pi_1=\{V_1,\cdots,V_s, W\}$ where $W$ contains $l_1$ and $\pi'=\{V_1',\cdots,V_t', W'\}$ where $W'$ contains $l_1$. Let $\pi=\{V_1,\cdots,V_s, V_1'\cdots,V_t', W\cup W'\}$.  Then, we need to show that $\pi\in IBNC(\chi)$. 

Firstly, we  show that $\pi$ is a $\chi$-interval partition.  Suppose that $i<l_k<j$ for some $k\geq 1$ and $i,j$ are in a same block of $\pi$. Notice that $j>l_1$, thus $j\in[l_1,n]$. We have the following two cases:

1. If $i\geq l_1$, then $i, l_k, j$ must be in a same block of $\pi'$ since $\pi'\in IBNC(\chi')$. In this case, $i, l_k, j$  are in a same block of $\pi$. 

2. If $i<l_1$, then $i\in W\cup W' $ which is the only possible block of $\pi$ contains such a pair $i,j$. Therefore, $l_1$ and $j$ are in $W'$. If $k=1$, we are done.   If $k>1$, then $l_k\in W'$ since $\pi'\in IBNC(\chi')$ and $l_1,j\in W'$. Thus $i,l_k,j$ are in a same block of $\pi$ which shows that $\pi$ is a $\chi$-interval partition.

Secondly, we turn to show that $\pi$ is a $\chi$-noncrossing partition. Let $i\prec_\chi j\prec_\chi k\prec_\chi l$ such that $i,k$ are in a same block of $\pi$ and $j,l$ are in a same block of $\pi$. We consider the following cases:

1.  If $l_1\leq \min\{i,j,k,l\}$(or  $l_1\geq \max\{i,j,k,l\}$), then $i,j,k,l$ must be in a same of $\pi_1$ (rep. $\pi'$) since $\pi_1\in IBNC(\chi_1)$ (rep. $\pi'\in IBNC(\chi')$). In this case, $i,j,k,l$ are in a same block of $\pi$.  

2.  If $l_1$ lies between $i, j$ with respect to the natural order, then $i,j$ are contained in the block $W\cup W'$. Therefore $i,j,k,l$ are contained in the block  $W\cup W'$. The same if $l_1$ lies between $i,l$ or $k,j$ or $k,l$.

3.  If $l_1$ lies between $i, k$ , then $i,k\in W\cup W'$. If  $l_1$ also lies between $j,l$, then $j,l\in W\cup W'$.  We are done.  Suppose that   $l_1$ does not lie between $j,l$, then $j,l< l_1$ or $j,l> l_1$.  Then, we have the following cases:

a) $j,l>l_1$, then we have $l_1\prec_{\chi} j$.  Notice that $j\prec_{\chi} k\prec_{\chi} l\prec_{\chi} $, we have  $l_1\prec_{\chi}  k\prec_{\chi} i_{t-1}$. Therefore, $l_1\prec_{\chi} j\prec_{\chi}k\prec_{\chi}l$. However,  $l_1,k$ are in a same block and $j,l$ are in a same block.  It follows that $l_1,j,k,l$ are in a same block since $\pi'$ is a $\chi'$-noncrossing partition. Therefore, $i,j,k,l$ are in the same block $W\cup W'$ of $\pi$.
$$\begin{tikzpicture}[baseline]
\draw[thick, dashed] (-1,0) -- (-1, -1);
\draw[thick, dashed] (1,-1) -- (1,0);
      \draw[thick, dashed]  (-2, -1) -- (2,-1) ;
\draw[thick, dashed] (-1,-1) arc (180:360:10mm);
   \node [draw, shade, circle, ball color=white, inner sep=0.08cm] at (-1, -1) {};  
    \node[left] at (-1,-1) {$l_1$};
         \node[left] at (-1,-0.5) {$i$};	
           \node[left] at (-0.707,-1.707) {$j$};		
           \node[below] at (0,-2) {$k$};		
           \node[right] at (0.707,-1.707) {$l$};		
 
			\node[right] at (-1.1,-3) {Case a)};
	\end{tikzpicture}
$$

b) If $j,l<l_1$ and $ j\prec_{\chi} l_1 $, then $i<l_1$ since $i\prec_{\chi} j$. Since $l_1$ lies between  since $j,k$, we have $k\geq l_2$. Since $l<l_1$ and  $k\prec_{\chi} l$,  $l\in\chi^{-1}(r)\cap[1,l_1]$. Therefore, we have $i\prec_{\chi} j\prec_{\chi} l_1\prec_{\chi} l$.  Since $i,l_1$ are in the same block $W$, $j,l$ are in a same block of $\pi_1$  and $\pi_1$ is a a $\chi_1$-noncrossing partition, $i,j,l_1,l$ are in $W$.  It follows that $ i,j,l_1,k,l$ are contained in $W\cup W' $.

$$\begin{tikzpicture}[baseline]
\draw[thick, dashed] (-1,1) -- (-1, -1);
\draw[thick, dashed] (1,-1) -- (1,1);
      \draw[thick, dashed]  (-2, -1) -- (2,-1) ;
\draw[thick, dashed] (-1,-1) arc (180:360:10mm);
   \node [draw, shade, circle, ball color=white, inner sep=0.08cm] at (-1, -1) {};  
    \node[left] at (-1,-1) {$l_1$};
    \node[left] at (-1,0.5) {$i$};
         \node[left] at (-1,-0.5) {$j$};	
       
           \node[below] at (0,-2) {$k$};		
            \node[right] at (1,-0.5)  {$l$};		
 
			\node[right] at (-1.1,-3) {Case b)};
	\end{tikzpicture}
$$
c) If $j,l<l_1$ and $ j\prec_{\chi} l_1 $, then $k\in \chi^{-1}\{r\}\cap [1,l_1]=\{i_{t}>i_{p+2}>\cdots>i_n\} $. Therefore,  $k<l_1$ and $i\geq l_1$. Since $i\prec_{\chi} j$, we have $l_1\prec_{\chi} j.$ Thus $l_1\prec_{\chi}j\prec_{\chi}k\prec_{\chi}l$.   $l_1,j,k,l$ are in $W$ since $\pi_1$ is a a $\chi_1$-noncrossing partition. In this case, we have that $ i,j,l_1,k,l$ are contained in $W\cup W' $.
$$\begin{tikzpicture}[baseline]
\draw[thick, dashed] (-1,1) -- (-1, -1);
\draw[thick, dashed] (1,-1) -- (1,1);
      \draw[thick, dashed]  (-2, -1) -- (2,-1) ;
\draw[thick, dashed] (-1,-1) arc (180:360:10mm);
   \node [draw, shade, circle, ball color=white, inner sep=0.08cm] at (-1, -1) {};  
    \node[left] at (-1,-1) {$l_1$};
    \node[right] at (1,0.5) {$l$};
         \node[right] at (1,0) {$k$};	
       
           \node[below] at (0,-2) {$i$};		
            \node[right] at (1,-0.5)  {$j$};		
 
			\node[right] at (-1.1,-3) {Case c)};
	\end{tikzpicture}
$$

4. Similarly, if $l_1$ lies between $j,l$ we also have that $i,j,k,l$ are in a same block of $\pi.$
Therefore, $\pi$ is a $\chi$-noncrossing partition.
The proof is complete.

\end{proof}

\begin{proposition}\label{lattice isomorphism} \normalfont 
 Let  $\alpha_1':IBNC(\chi)\rightarrow NC(\chi_1)\times IBNC(\chi')$ such that 
$$\alpha'_1(\pi)=(\alpha_1(\pi), \alpha'(\pi)) $$
and $\alpha:IBNC(\chi)\rightarrow NC(\chi_1)\times NC(\chi_2)\times\cdots \times NC(\chi_{m+1})$ such that 
$$\alpha(\pi)=(\alpha_1(\pi), \cdots,\alpha_{m+1}(\pi)).$$
Then $\alpha'_1$ and $\alpha$ are lattice isomorphisms, $IBNC(\chi)$ is a lattice with respect to the reverse refinement order $\leq$ on partitions.
\end{proposition}
\begin{proof}
The statement for $\alpha$ is an induction argument from $\alpha_1'$.  In \cite{Liu3}, it is shown that restrictions of partitions on subintervals preserve the reversed refinement order.   Therefore,  we only need to show that $\alpha_1'$ is a bijection which  follows  Lemma \ref{injective} and Lemma \ref{Surjective}. The proof is complete
\end{proof}

\begin{proposition}\normalfont   Let $\pi=\{V_1,\cdots,V_t\}\in IBNC(\chi)$ and  $\sigma$  be a partition of $[n]$ such that $\sigma\leq \pi$ with respect to the reversed refinement order. Then, $\sigma\in IBNC(\chi)$ if and only if $\sigma|V_s\in IBNC(\chi|_{V_s})$ for all $s=1,\cdots,t$.
\end{proposition}
\begin{proof}
By  Proposition 5.8 in \cite{Liu3},  $\sigma$ is a $\chi$-interval partition  if and only if  $\sigma|V_s$ is a $\chi|_{V_s}$-interval partition  for all $s=1,\cdots,t$. If we consider $\prec_{\chi}$ be the only order on $[n]$, by Theorem 9.29 in \cite{NS}, $\sigma$ is a $\chi$-noncrossing partition  if and only if  $\sigma|V_s$ is a $\chi|_{V_s}$-noncrossing partition  for all $s=1,\cdots,t$. The statement follows.
\end{proof}

We see that the set of interval-bi-noncrossing partitions, which are finer than a given interval-bi-noncrossing partition $\pi$, is uniquely determined by the IBNC-partitions with respect to  the restrictions of $\chi$ to the blocks of $\pi$. 
Therefore, we have the following decomposition property.

\begin{proposition}\label{canonical isomorphism}\normalfont 
 Let $\pi=\{V_1,\cdots,V_t\}\in IBNC(\chi)$.  Then 
 $$[0_n, \pi]\cong IBNC(\chi|_{V_1})\times\cdots\times IBNC(\chi|{V_s}),$$
 where  $0_n$ is the partition of $[n]$ into $n$ blocks, and $[0_n,\pi]$ is the interval $\{\sigma\in IBNC(\chi): 0_n\leq \sigma\leq \pi\}.$
\end{proposition}

\section{M\"obius functions}
In this section,  we  study some properties of  M\"obius functions on $IBNC(\chi)$.  

Let $L$ be a finite lattice. We denote by
$$L^{(2)}=\{(a,b)|b,a\in L,\,\, a\leq b\}$$
the set of ordered pairs of elements in $L$.\\
Given two functions $f,g:L^{(2)}\rightarrow \C$,  their convolution $f*g$ is given by:
$$f*g(a,b)=\sum\limits_{ \substack{c\in L\\ a\leq c\leq b}} f(a,c)g(c,b).$$
It is shown by Rota \cite{Rota}, the following three special functions on $L^{(2)}$ always exist.
\begin{itemize}
\item The \emph{delta} function defined as 
$$\delta(a,b)=\left\{\begin{array}{ll}
1,\,\,\,\, &\text{if}\,\, a=b,\\
0,&\text{otherwise.}
\end{array}\right.
$$
\item The \emph{zeta} function $\zeta$  defined as
$$\zeta(a,b)=\left\{\begin{array}{ll}
1,\,\,\,\, &\text{if}\,\, a\leq b,\\
0,&\text{otherwise.}
\end{array}\right.
$$
\item By Proposition 1 in \cite{Rota}, there is a  function $\mu$ on $L^{(2)}$ such that 
$$\mu*\zeta=\zeta*\mu=\delta. $$
$\mu$ is called the \emph{M\"obius function} of $L^{(2)}$.
\end{itemize}

Here, $\delta$ is the unit  with respect to the convolution $*$ and $ \mu$ is the inverse of $\zeta$ with respect to $*$.  Given lattices $L_1,\cdots, L_m$, their direct product $L=L_1\times\cdots\times L_m$ is also a lattice with respect to the order such that
$(a_1,...,a_m)\leq(b_1,...,b_m)$ if and only if $a_i\leq b_i$ for all $i$. 
 It is obvious that $L^{(2)}=L_1^{(2)}\times\cdots\times L_m^{(2)}$.
For each $i$, let $f_i$ be a $\C$-valued function on $L_i^{(2)}$. The product $f=\prod\limits_{i=1}^mf_i$ is a function on $L^{(2)}$ defined as follows:
$$f((a_1,...,a_m),(b_1,...,b_m))= \prod\limits_{i=1}^mf_i(a_i,b_i),$$
for all  $(a_1,...,a_m),(b_1,...,b_m)\in L^{(2)}$. The following result is Lemma 6.1 in \cite{Liu3}.

\begin{lemma}\label{product of Mobius function} \normalfont  Let $L_1,\cdots,L_m$ be finite lattices. For each $i$, let $\delta_i$, $\zeta_i$, $\mu_i$ be the delta function, the zeta function and the M\"obius function on $L_i^{(2)}$  respectively. Then $\bar\delta=\prod\limits_{i=1}^m\delta_i$, $\bar\zeta=\prod\limits_{i=1}^m\zeta_i$ and $\bar\mu=\prod\limits_{i=1}^m\mu_i$ are  the delta function, the zeta function and the M\"obius function on $L_1^{(2)}\times\cdots\times L_m^{(2)}$ respectively. 
\end{lemma}

Therefore, we have the following result.

\begin{proposition}\normalfont  Let $\alpha:IBNC(\chi)\rightarrow NC(\chi_1)\times NC(\chi_2)\times\cdots \times NC(\chi_{m+1})$ be the lattice isomorphism in Proposition \ref{lattice isomorphism}. 
Let $\bar\delta$, $\bar\zeta$, $\bar\mu$ be the delta function, the zeta function and the M\"obius function of $NC(\chi_1)\times NC(\chi_2)\times\cdots \times NC(\chi_{m+1})$.  Then $\delta_{IBNC(\chi)}=\bar\delta\circ\alpha$, $\zeta_{IBNC(\chi)}=\bar\zeta\circ\alpha$, $\mu_{IBNC(\chi)}=\bar\mu\circ\alpha$ are the delta function, the zeta function and the M\"obius function on ${IBNC(\chi)}$.
\end{proposition}

  Let $(\sigma,\pi)\in IBNC(\chi)^{(2)}$ such that $\alpha(\sigma)=(\sigma_1,\cdots,\sigma_{m+1}),\alpha(\pi)=(\pi_1,\cdots,\pi_{m+1})\in NC(\chi_1)\times NC(\chi_2)\times\cdots \times NC(\chi_{m+1})$.  Then, we have 
$$\mu_{IBNC(\chi)}(\sigma,\pi)=\prod\limits_{i=1}^{m+1} \mu_{NC(\chi_i)}(\sigma_i,\pi_i).$$
For convenience, we let $\mu({\emptyset},1_{\emptyset})=1$. Given a partition $\pi\in IBNC(\chi)$ and a blcok $V\in \pi$,  we set  $\tilde{\alpha}_i(V)=V\cap [l_{i-1},l_i]$, $i=1,\cdots, m+1$.  

\begin{lemma}\label{Mobius transform1}\normalfont  Let $\pi=\{V_1,\cdots,V_t\}\in IBNC(\chi)$ and $\sigma\in IBNC(\chi)$ such that $\sigma\leq \pi$. Then,
$$\mu_{IBNC(\chi)}(\sigma|_{V_s},1_{V_s})=\prod\limits_{i=1}^{m+1} \mu_{IBNC(\chi_i|_{V_s})}(\sigma_i|\talpha_i(V_s),1_{\talpha_i(V_s)}),$$
where $1_{\talpha_i(V_s)}$ is the partition of $\talpha_i(V_s)$ into one block.
\end{lemma}
\begin{proof}
By Proposition \ref{canonical isomorphism}, we have  $$IBNC(\chi|_{V_s})=IBNC(\chi|_{\talpha_1(V_s)})\times IBNC(\chi|_{\talpha_2(V_s)})\times\cdots \times IBNC(\chi|_{\talpha_{m+1}(V_s)}).$$ The statement follows from Lemma \ref{product of Mobius function}.
 \end{proof}
Further more, we have the following result.
\begin{lemma}\label{Mobius transform}\normalfont  Let $\pi=\{V_1,\cdots,V_t\}\in IBNC(\chi)$ and $\sigma\in IBNC(\chi)$ such that $\sigma\leq \pi$. Then,
$$\mu_{IBNC(\chi)}(\sigma,\pi)=\prod\limits_{s=1}^t\prod\limits_{i=1}^{m+1} \mu_{IBNC{\chi_i|_{V_s}}}(\sigma_i|\talpha_i(V_s),1_{\talpha_i(V_s)}).$$
\end{lemma}

By  Lemma \ref{Mobius transform1} and \ref{Mobius transform}, we have the follow proposition.
\begin{proposition}\label{Mobius transform2} \normalfont
Let $\pi=\{V_1,\cdots,V_t\}\in IBNC(\chi)$ and $\sigma\in IBNC(\chi)$ such that $\sigma\leq \pi$. Then,
$$\mu_{IBNC}(\sigma,\pi)=\prod\limits_{s=1}^t\mu_{IBNC(\chi|_{V_s})}(\sigma|_{V_s},1_{V_s}).$$
\end{proposition}

\section{ Vanishing cumulants condition for free-free-Boolean independence}
In this section, we introduce the notion of free-free-Boolean cumulants and  show that the  vanishing of mixed free-free-Boolean cumulants is equivalent to our free-free-Boolean independence.
\subsection{Free-Boolean cumulants}
Let $(\A,\phi)$ be a noncommutative probability space. For $n\in \mathbb{N}$, let $\phi^{(n)}$ be the \emph{$n$-linear} map from $\underbrace{\A\otimes\cdots\otimes \A}_{\text{n times}}$ to $\C$  defined as
$$\phi^{(n)}(z_1,\cdots,z_n)=\phi(z_1\cdots z_n),$$
where $z_1,...,z_n\in \A.$\\
Let $V=\{l_1<l_2<\cdots< l_k\}$ be a subset of $ [n]$. Then we define
$$\phi_{V}(z_1,\cdots,z_n)=\phi(z_{l_1}z_{l_2}\cdots z_{l_k}). $$
Let $\pi=\{V_1,\cdots,V_s\}$ be a partition on $[n]$.  Then we  define an $n$-linear map $\phi_{\pi}: \underbrace{\A\otimes\cdots\otimes \A}_{\text{n times}}\rightarrow \C$ as follows:
$$\phi_\pi(z_1,\cdots,z_n)=\prod\limits_{t=1}^s \phi_{V_t}(z_1,\cdots,z_n).$$ 
For example,  let $n=8$ and $\pi=\{\{1,5,8\},\{2,3,4\},\{6,7\}\}$. Then,
$$\phi_{\pi}(z_1,\cdots,z_8)=\phi(z_1z_5z_8)\phi(z_6z_7)\phi(z_2z_3z_4).$$
\begin{definition}  \normalfont Given $\chi$ and $\pi\in IBNC(\chi)$, the \emph{free-free-Boolean cumulant} $\kappa_{\chi,\pi}$ is an $n$-linear map defined as follows:
$$\kappa_{\chi,\pi}(z_1,\cdots,z_n)
=\sum\limits_{\substack{\sigma\leq \pi\\ \sigma\in IBNC(\chi)}}\mu_{IBNC}(\sigma,\pi)\phi_{\sigma}(z_1,\cdots,z_n) .$$ 
\end{definition}

\begin{theorem}\label{multiplicative property} \normalfont  Let $\pi=\{V_1,\cdots,V_t\}\in IBNC(\chi)$ and $z_1,\cdots,z_n$ be  random variables in a  probability space $(\A,\phi)$. Then
$$\kappa_{\chi,\pi}(z_1,\cdots,z_n)=\prod\limits_{s=1}^t \kappa_{\chi|_{V_s},1_{V_s}}(z_1,\cdots,z_n). $$
\end{theorem}
\begin{proof} By Lemma \ref{Mobius transform}, we have 
$$
\begin{array}{rcl}
&&\kappa_{\chi,\pi}(z_1,\cdots,z_n)\\

&=&\sum\limits_{\substack{\sigma\leq \pi\\ \sigma\in IBNC(\chi)}}\mu_{IBNC}(\sigma,\pi)\phi_{\sigma}(z_1,\cdots,z_n)\\

&=&\sum\limits_{\substack{\sigma\leq \pi\\ \sigma\in IBNC(\chi)}}\left[\prod\limits_{i=1}^m\prod\limits_{s=1}^{t} \mu_{IBNC(\chi_i|_{V_s})}(\sigma_i|\talpha_i({V_s}),1_{\talpha_i(V_s)})\right]\left[\prod\limits_{s=1}^{t}\phi_{\sigma|_{V_s}}(z_1,\cdots,z_n)\right]\\

&=&\sum\limits_{\substack{\sigma\leq \pi\\ \sigma\in IBNC(\chi)}}\prod\limits_{s=1}^t\left[\prod\limits_{i=1}^{m} \mu_{IBNC(\chi_i|_{V_s})}(\sigma_i|\talpha_i(V_s),1_{\talpha_i(V_s)})\right]\left[\prod\limits_{s=1}^{t}\phi_{\sigma|_{V_s}}(z_1,\cdots,z_n)\right]\\

&=&\sum\limits_{\substack{\sigma\leq \pi\\ \sigma\in IBNC(\chi)}}\prod\limits_{s=1}^t\left[\prod\limits_{i=1}^{m} \mu_{IBNC(\chi_i|_{V_s})}(\sigma_i|\talpha_i(V_s),1_{\talpha_i(V_s)})\phi_{\sigma|_{V_s}}(z_1,\cdots,z_n)\right]\\

&=&\sum\limits_{\substack{\sigma\leq \pi\\ \sigma\in IBNC(\chi)}}\prod\limits_{s=1}^t\left[\mu_{IBNC(\chi|_{V_s})}(\sigma|_{V_s},1_{V_s})\phi_{\sigma|_{V_s}}(z_1,\cdots,z_n)\right],\\

\end{array}
$$
where the last equality follows Lemma \ref{Mobius transform}. By corollary \ref{canonical isomorphism}, we have that 
$$
\begin{array}{rcl}
&&\sum\limits_{\substack{\sigma\leq \pi\\ \sigma\in IBNC(\chi)}}\prod\limits_{s=1}^t\left[\mu_{IBNC}(\sigma|_{V_s},1_{V_s})\phi_{\sigma|_{V_s}}(z_1,\cdots,z_n)\right]\\
&=&\prod\limits_{s=1}^t\left[\sum\limits_{\substack{ \sigma|_{V_s}\in IBNC(\chi|_{V_s})}}\mu_{IBNC}(\sigma|_{V_s},1_{V_s})\phi_{\sigma|_{V_s}}(z_1,\cdots,z_n)\right]\\
&=&\prod\limits_{s=1}^t \kappa_{\chi|_{V_s},1_{V_s}}(z_1,\cdots,z_n),
\end{array}
$$
thus the proof is complete.
\end{proof}

\begin{definition} \normalfont 
Let $\{(\mathcal{A}_{i,\ell}, \mathcal{A}_{i,r},\mathcal{A}_{i,c})\}_{i \in I}$ be a family of triples of faces in a  probability space $(\mathcal{A}, \varphi)$. We say that the family $\{(\mathcal{A}_{i,\ell}, \mathcal{A}_{i,r},\mathcal{A}_{i,c})\}_{i \in I}$ is  \emph{combinatorially free-free-Boolean independent} if 
$$\kappa_{\chi,1_n}(z_1,\cdots, z_n)=0 $$
whenever  $\chi: [n]\to \{\ell, r,c\}$, $\omega : [n]\to I$,  $z_k\in\A_{\omega(k),\chi(k)}$ and $\omega$ is not a constant.
\end{definition}

\begin{proposition}\normalfont 
Let $\{(\mathcal{A}_{i,\ell}, \mathcal{A}_{i,r},\mathcal{A}_{i,c})\}_{i \in I}$ be a family of triples of faces in a  probability space $(\mathcal{A}, \varphi)$. 
Then $\kappa_{\chi,1_n}$ has the following  cumulant property:
$$ \kappa_{\chi,1_n}(z_{1,1}+z_{2,1},\cdots, z_{1,n}+z_{2,n})=\kappa_{\chi,1_n}(z_{1,1},\cdots, z_{1,n})+\kappa_{\chi,1_n}(z_{2,1},\cdots, z_{2,n})$$
whenever $\omega_1,\omega_2:[n]\rightarrow I$, $\chi:[n]\rightarrow \{\ell, r,c\} $, $z_{1,k}\in\A_{\omega_1(k),\chi(k)}$, $z_{2,k}\in\A_{\omega_2(k),\chi(k)}$  and $\omega_1([n])\cap\omega_2([n])=\emptyset$.
\end{proposition}
\begin{proof}
By direct calculations, we have
$$\kappa_{\chi,1_n}(z_{1,1}+z_{2,1},\cdots, z_{1,n}+z_{2,n})=\sum\limits_{i_1,...i_n\in\{1,2\}}\kappa_{\chi,1_n}(z_{i_1,1},\cdots, z_{i_n,n}).$$
Since $\{(\mathcal{A}_{i,\ell}, \mathcal{A}_{i,r})\}_{i \in I}$ are combinatorially free-Boolean independent, by the preceding definition, we have 
$$\kappa_{\chi,1_n}(z_{i_1,1},\cdots, z_{i_n,n})=0$$
if $i_k\neq i_j$ for some $j,k\in[n]$. The result follows.
\end{proof}

\subsection{Free-Boolean is equivalent to combinatorially free-Boolean }In this subsection,  we will prove the following main theorem:

\begin{theorem}\label{Main} \normalfont 
Let $\{\A_{i,\ell}, \A_{i,r},\A_{i,c}\}_{i \in I}$ be a family of triples of faces in a probability space $(\mathcal{A}, \phi)$. $\{\A_{i,\ell}, \A_{i,r},\A_{i,c}\}_{i \in I}$ are free-Boolean independent if and only if they are combinatorially free-free-Boolean.
\end{theorem}
It is sufficient to show  to  that mixed  moments  are uniquely determined by lower mixed moments in the same way for combinatorially free-free-Boolean independence and  free-free-Boolean independence.
  By Proposition 10.6 in \cite{NS} and Theorem \ref{multiplicative property}, we have  the following result.
\begin{lemma}\label{Moments-cumulant}  \normalfont  Let  $z_1,\cdots,z_n$ be  random variables in a noncommutative probability space $(\A,\phi)$.  Then
$$\phi(z_1\cdots z_n)=\sum\limits_{\pi\in IBNC(\chi)} \kappa_{\chi,\pi} (z_1,\cdots, z_n).$$
\end{lemma}

For combinatorially free-Boolean independent random variables, we have the following result.
\begin{lemma}\label{c-fb}\normalfont 
Let $\{\A_{i,\ell}, \A_{i,r},\A_{i,c}\}_{i \in I}$ be a family of  combinatorially free-Boolean independent  triples of faces in  a noncommutative probability space $(\A,\phi)$.  Assume that $z_k\in\A_{\omega(k),\chi(k)}$, where $\omega:[n]\rightarrow I$, $\chi:[n]\rightarrow \{\ell, r,c\}$.  Let $\epsilon=\ker \omega$. Then, 
\begin{equation}\label{recursive relation}
\phi(z_1\cdots z_n)=\sum\limits_{\sigma\in IBNC(\chi)} \left(\sum\limits_{\substack{ \pi\in IBNC(\chi)\\ \sigma\leq \pi\leq \epsilon}}\mu_{IBNC}(\sigma,\pi)\right)\phi_{\sigma} (z_1,\cdots, z_n). \tag{$\bigstar$}
\end{equation}
\end{lemma}
\begin{proof}
By Lemma \ref{Moments-cumulant}, we have 
$$\phi(z_1\cdots z_n)=\sum\limits_{\pi\in IBNC(\chi)} \kappa_{\chi,\pi} (z_1,\cdots ,z_n).$$
For each $\pi\in IBNC(\chi)$, assume that $\pi=\{V_1,\cdots,V_t\}$. By Theorem \ref{multiplicative property}, we have  
$$\kappa_{\chi,\pi} (z_1\cdots z_n)=\prod\limits_{s=1}^t \kappa_{\chi|_{V_s},1_{V_s}}(z_1,\cdots,z_n). $$
Since the family $\{\A_{i,\ell}, \A_{i,r},\A_{i,c}\}_{i \in I}$ is combinatorially free-free-Boolean independent,
$$\kappa_{\chi|_{V_s},1_{V_s}}(z_1,\cdots,z_n)=0$$
if $\omega$ is not a  constant on ${V_s}$.   It follows that $\kappa_{\chi,\pi}(z_1,\cdots,z_n)\neq 0$ only if $\omega$ is  a constant on $|_{V_s}$ for all $s$, which implies that $V_s$ is contained in a block of $\epsilon$ for all $s$, i.e.,  $\pi\leq \epsilon$.
 Therefore, we have
$$
\begin{array}{rcl}
\phi(z_1\cdots z_n)
&=&\sum\limits_{\pi\in IBNC(\chi),\pi\leq \epsilon} \kappa_{\chi,\pi} (z_1,\cdots, z_n)\\
                               &=&\sum\limits_{\pi\in IBNC(\chi),\pi\leq \epsilon} \left(\sum\limits_{\substack{ \sigma\in IBNC(\chi)\\ \sigma\leq \pi}}\mu_{IBNC}(\sigma,\pi)\phi_{\sigma} (z_1,\cdots, z_n)\right)\\
                               &=&\sum\limits_{\sigma\in IBNC(\chi)} \left(\sum\limits_{\substack{ \pi\in IBNC(\chi)\\ \sigma\leq \pi\leq \epsilon}}\mu_{IBNC}(\sigma,\pi)\right)\phi_{\sigma} (z_1,\cdots ,z_n).\\
\end{array}
$$
This finishes the proof.
\end{proof}

Now, we suppose that the family $\{\A_{i,\ell}, \A_{i,r},\A_{i,c}\}_{i \in I}$ is  free-free-Boolean independent in $(\A,\phi)$.
We  assume that $z_k\in\A_{\omega(k),\chi(k)}$, where $\omega:[n]\rightarrow I$, $\chi:[n]\rightarrow \{\ell, r,c\}$. Let $\epsilon$ be the kernel of $\omega$. Let $\chi_1$ and $\epsilon_1$ be the restriction of $\chi$ and $\epsilon$ to the first interval $\{1,\cdots, l_1\}$ respectively.  
Let $\chi_1'$ and $\epsilon_1'$ be the restrictions of $\chi$ and $\epsilon$ to the first interval $\{l_1,\cdots, n\}$ respectively.  We need to show that  the the mixed moments $\phi(z_1\cdots z_n)$ can be determined in the same way as in Lemma \ref{c-fb}. 

It is sufficient to consider the case that $\A=\LL(\X)$,  $\mathcal{A}_{i,\ell} =\lambda_i(\mathcal{L}(\X_i))$, $\mathcal{A}_{i,r} =\lambda_i(\mathcal{L}(\X_i))$ and $\mathcal{A}_{i,c} =P_{i}\lambda_i(\mathcal{L}(\X_i))P_{i}$, where   $(\X_i, \mrx_i,\xi_i)_{i\in I}$ is a family of vector spaces with specified vectors, $(\mathcal{X}, \mrx,\xi)$ is their  reduced free product and  $\phi=\phi_\xi$ is the functional associated with $\xi$ on $\X$.

As the free-Boolean case in \cite{Liu3}, we prove the mixed moments formula $(\bigstar)$ in Lemma \ref{c-fb}   by  induction on the number of elements  of  $\chi^{-1}(\circ)\cap[2,n-1]$.

\begin{lemma}\label{replace z_n} \normalfont  Let $z_1\in \A_{i,c}$ and $z_2\in \A_{j,c}$ for some $i,j\in \I$. Then, there exist $T_1\in \A_{i,\ell}$ and $T_2\in\A_{j,\ell}$ such that 
$$\phi(z_1z z_n)= \phi(T_1zT_2).$$
for all $z\in\A.$ Moreover, $T_1$ and $T_2$ are uniquely determined by $z_1$ and $z_2$ respectively.
 \end{lemma}
\begin{proof}
By definition of $\A_{i,c}$ and $\A_{j,c}$,  $z_1=P_iT_1P_i$, $z_2=P_jT_2P_j$ for some $T_1\in \A_{i,\ell}$, $T_2\in\A_{j,\ell}$. By Proposition \ref{simple Boolean}, $z_1=P_iT_1$, $z_2=T_2P_j$ Let $p$ be the projection onto $\C\xi$ and vanishes on $\mrx$. Then $pP_j=p$, thus
$$pz_1z z_n\xi= pP_iT_1zT_2P_j\xi =pT_1zT_2\xi.$$
The result follows from the definition of $\phi$.
\end{proof}

The same we have the following statement for $A_{i,c}$ and $A_{i,r}$.
\begin{lemma}\label{replace z_1} Let $z_1\in \A_{i,c}$ and $z_2\in \A_{j,c}$ for some $i,j\in \I$. Then, there exist $T_1\in \A_{i,r}$ and $T_2\in\A_{j,r}$ such that 
$$\phi(z_1z z_n)= \phi(T_1zT_2).$$
for all $z\in\A.$ Moreover, $T_1$ and $T_2$ are uniquely determined by $z_1$ and $z_2$ respectively.
\end{lemma}
By Lemma \ref{replace z_n} and Lemma \ref{replace z_1}, we have the following result.
\begin{corollary}\label{replace z_r}
 Let $z_1\in \A_{i,r}$ and $z_2\in \A_{j,r}$ for some $i,j\in \I$. Then, there exist $T_1\in \A_{i,\ell}$ and $T_2\in\A_{j,\ell}$ such that 
$$\phi(z_1z z_n)= \phi(T_1zT_2).$$
for all $z\in\A.$ Moreover, $T_1$ and $T_2$ are uniquely determined by $z_1$ and $z_2$ respectively.
\end{corollary}

In this case of $|\chi^{-1}(\circ)\cap[2,n-1]|=0$,  $\chi$ can be $\circ$ only at $1$ and $n$.  Thus, $IBNC(\chi)=NC(\chi)$ which is isomorphic to the set of bi-noncrossing partitiona $BNC(\bar\chi)$ defined in \cite{CNS} where
$$\bar\chi(i)=\left\{\begin{array}{ll}
\ell&\text{if}\quad\chi(i)=\ell,c,\\
r&\text{if}\quad\chi(i)=r.\\
\end{array}\right.
$$
Therefore, we have the following result.

\begin{lemma}\label{Induction 1} \normalfont When $|\chi^{-1}(\circ)\cap[2,n-1]|=0$, we have 
$$\phi(z_1\cdots z_n)=\sum\limits_{\sigma\in IBNC(\chi)} \left(\sum\limits_{\substack{ \pi\in IBNC(\chi)\\ \sigma\leq \pi\leq \epsilon}}\mu_{IBNC}(\sigma,\pi)\right)\phi_{\sigma} (z_1\cdots z_n).\\$$
\end{lemma}
\begin{proof}
By Lemma \ref{replace z_n}, Lemma \ref{replace z_n} and Corollary \ref{replace z_r}, we have 
$$ \phi(z_1\cdots z_n)=\phi(T_1z_2\cdots z_{n-1}T_2),$$
for $T_1\in \A_{\omega(1),\ell}$ and $ T_2\in\A_{\omega(n),\ell}$. 
Of course, if $z_1\in \A_{\omega(1),\ell}$, then the just let $T_1=z_1.$ The same to $z_n$. 
Let $\bar{\chi}:[n]\rightarrow\{\ell,r,c\}$ such that $\bar{\chi}(k)=\chi(k)$ for $k=2,\cdots,n-1$ and $\bar{\chi}(1)=\bar{\chi}(n)=\ell$. Since $\chi^{-1}(\circ)\cap[2,n-1]=\emptyset$, $\bar{\chi}$ is a map from $[n]$ to $\{\ell,r\}$. Notice that  the family $\{(\mathcal{A}_{i,\ell},\mathcal{A}_{i,r})\}$ is bi-freely independent,  by Theorem 2.3.2 in \cite{CNS1}, we have 
 $$
\begin{array}{rcl}
&&\phi(T_1z_2\cdots z_{n-1}T_2)\\
   &=&\sum\limits_{\sigma\in BNC(\bar\chi)} \left(\sum\limits_{\substack{ \pi\in BNC(\bar\chi)\\ \sigma\leq \pi\leq \epsilon}}\mu_{BNC(\bar\chi)}(\sigma,\pi)\right)\phi_{\sigma} (T_1,z_2,\cdots ,z_{n-1},T_2)\\
   &=&\sum\limits_{\sigma\in IBNC(\chi)} \left(\sum\limits_{\substack{ \pi\in IBNC(\chi)\\ \sigma\leq \pi\leq \epsilon}}\mu_{IBNC}(\sigma,\pi)\right)\phi_{\sigma} (z_1,\cdots ,z_n).
\end{array} 
$$
By  Lemma \ref{replace z_n}, Lemma \ref{replace z_n} and Corollary \ref{replace z_r}, we replace $z_1$ and $z_n$ back in last equality. The proof is complete.
\end{proof}

Now we are ready to prove our main theorem. 
\begin{proof}[Proof of Theorem \ref{Main}]
Suppose that  Equation $(\bigstar)$ in Lemma \ref{c-fb} holds  whenever $|\chi^{-1}(\circ)\cap[2,n-1]|=m-2$.  For $|\chi^{-1}(\circ)\cap[2,n-1]|=m-1$,  we can assume $\chi^{-1}(\circ)=\{0<l_1<\cdots<l_{m-1}<n\}$, $l_0=1$ and $l_m=n$.

Let $Z_1=\prod\limits_{i=l_1}^n z_i$.  Since  the range of $z_{l_1}$ is  $\C\xi\oplus \mrx_{\omega(l_1)}$,  $Z_1$ is a linear map from $\C\xi\oplus \mrx_{\omega(l_1)}$  to $\C\xi\oplus \mrx_{\omega(l_1)}$.  Moreover,  $Z_1$ vanishes on all summands of $\X$ except for $\C\xi\oplus \mrx_{\omega(l_1)}$. Therefore,  $Z_1$ is an element in $P_{\omega(l_1)}\lambda_{\omega(l_1)}(\LL(\X_{\omega(l_1)}))P_{\omega(l_1)}$. \\
By Lemma \ref{Induction 1}, we have
$$
\begin{array}{rcl}
\phi(z_1\cdots z_n)&=&\phi(z_1\cdots z_{l_1-1}Z_1)\\
                               %&=&\sum\limits_{\pi_1\in IBNC(\chi_1),\pi_1\leq \epsilon_1} \Big(\sum\limits_{\substack{ \sigma_1\in BNC(l_1)\\ \sigma_1\leq \pi}}\mu(\sigma_1,\pi_1)\phi_{\sigma_1}(z_1\cdots z_{l_1-1}Z_{1})\Big)\\
                               &=&\sum\limits_{\sigma_1\in IBNC(\chi_1),\sigma_1\leq \epsilon_1} \left(\sum\limits_{\substack{ \sigma_1\in IBNC(\chi_1)\\ \sigma_1\leq \pi_1\leq \epsilon_1}}\mu_{IBNC}(\sigma_1,\pi_1)\right)\phi_{\sigma_1}(z_1,\cdots ,z_{l_1-1},Z_{1}).\\
\end{array}
$$
Fix $\sigma_1$,  let  $V$ be the block of $\sigma_1$ which contains $l_1$. Then,
$$ 
\begin{array}{rcl}
\phi_{\sigma_1}(z_1,\cdots ,z_{l_1-1},Z_{1})&=&\prod\limits_{W\in \sigma_1} \phi_{W}(z_1,\cdots, z_{l_1-1},Z_1)\\
&=&\left[\prod\limits_{W\neq V} \phi_{W}(z_1,\cdots, z_{l_1-1})\right]\phi_{V}(z_1,\cdots ,z_{l_1-1},Z_{1})\\
\end{array}
$$
the last equality follows from that  $l_1\not\in W$ whenever $W\neq V$. 

Let $Z^V_{l_1}=\prod\limits_{i\in V}z_i$, where the product is taken with the original order. Then,
$$ \phi_{V}(z_1,\cdots ,z_{l_1-1},Z_{1})=\phi(Z_{l_1} ^V z_{l_1+1}z_{l_1+2}\cdots z_n).$$
Notice that $|\chi^{-1}[l_1+1,n-1]|=m-2$ . By assumption, we have
$$\phi(Z_{l_1} ^V z_{l_1+1}z_{l_1+2}\cdots z_n)=\sum\limits_{\sigma'\in IBNC(\chi')} \left(\sum\limits_{\substack{ \pi'\in IBNC(\chi')\\ \sigma'\leq \pi'\leq \epsilon'}}\mu_{IBNC}(\sigma',\pi')\right)\phi_{\sigma'} (Z_{l_1} ^V ,z_{l_1+1},z_{l_1+2},\cdots ,z_n). $$

Suppose that $l_1\in V'\in\sigma'$. Then
$$ 
\begin{array}{rcl}
\phi_{\sigma'} (Z_{l_1} ^V, z_{l_1+1},z_{l_1+2},\cdots, z_n)&=&\prod\limits_{W\in \sigma'} \phi_{W}(Z_{l_1} ^V ,z_{l_1+1},z_{l_1+2},\cdots ,z_n)\\
&=&\left[\prod\limits_{W\neq V'} \phi_{W}((z_{l_1+1},\cdots, z_{n})\right]\phi_{V'}(Z_{l_1} ^V ,z_{l_1+1},z_{l_1+2},\cdots,z_n).\\
\end{array}
$$

Therefore,
$$\begin{array}{rcl}
&&\phi(z_1\cdots z_n)\\

&=&\sum\limits_{\substack{\sigma_1\in IBNC(\chi_1)\\\sigma_1\leq \epsilon_1}} \left(\sum\limits_{\substack{ \sigma_1\in BNC(\chi_1)\\ \sigma_1\leq \pi_1\leq \epsilon_1}}\mu_{IBNC}(\sigma_1,\pi_1)\right)\left[\prod\limits_{\substack{W\in \sigma_1\\W\neq V\ni l_1}} \phi_{W}(z_1,\cdots, z_{l_1-1})\right]\phi_{V}(z_1,\cdots, z_{l_1-1},Z_{1})\\

&=&\sum\limits_{\substack{\sigma_1\in IBNC(\chi_1)\\\sigma_1\leq \epsilon_1}} \left(\sum\limits_{\substack{ \sigma_1\in BNC(\chi_1)\\ \sigma_1\leq \pi_1\leq \epsilon_1}}\mu_{IBNC}(\sigma_1,\pi_1)\right)
\Bigg\{\left[\prod\limits_{\substack{W\in \sigma_1\\W\neq V\ni l_1}} \phi_{W}(z_1,\cdots, z_{l_1-1})\right]\\

&&\sum\limits_{\substack{\sigma'\in IBNC(\chi')\\\sigma'\leq \epsilon'}} \left(\sum\limits_{\substack{ \pi'\in IBNC(\chi')\\ \sigma'\leq \pi'\leq \epsilon'}}\mu_{IBNC}(\sigma',\pi')\right)\phi_{\sigma'} (Z_{l_1} ^V, z_{l_1+1},z_{l_1+2},\cdots ,z_n)\Bigg\} \\

&=&\sum\limits_{\substack{\sigma_1\in IBNC(\chi_1)\\\sigma_1\leq \epsilon_1}} \sum\limits_{\substack{\sigma'\in IBNC(\chi')\\\sigma'\leq \epsilon'}} \left(\sum\limits_{\substack{ \sigma_1\in IBNC(\chi_1)\\ \sigma_1\leq \pi_1\leq \epsilon_1}}\mu_{IBNC}(\sigma_1,\pi_1)\right)\left(\sum\limits_{\substack{ \pi'\in IBNC(\chi')\\ \sigma'\leq \pi'\leq \epsilon'}}\mu_{IBNC}(\sigma',\pi')\right)\\
&&\left\{[\prod\limits_{\substack{W\in \sigma_1\\W\neq V\ni l_1}} \phi_{W}(z_1,\cdots, z_{l_1-1})]\phi_{\sigma'} (Z_{l_1} ^V ,z_{l_1+1},z_{l_1+2},\cdots ,z_n)\right\} \\
\end{array}
$$

For fixed $\sigma_1$ and $\sigma'$, $\sigma=\alpha'^{-1}(\sigma_1,\sigma')\in IBNC(\chi)$ and
$$\mu_{IBNC}(\sigma,\pi)=\mu_{IBNC}(\sigma_1,\pi_1)\mu_{IBNC}(\sigma',\pi').$$
Since the blocks of $\sigma$ are exactly the blocks $W$, $W'$ and $V\cup V'$ in the preceding formula,  we have 
$$\left[\prod\limits_{\substack{W\in \sigma_1\\W\neq V\ni l_1}} \phi_{W}(z_1,\cdots, z_{l_1-1})\right]\phi_{\sigma'} (Z_{l_1} ^V ,z_{l_1+1},z_{l_1+2},\cdots ,z_n)=\phi_{\sigma}(z_1,\cdots,z_n).$$
It follows that
$$\phi(z_1\cdots z_n)=\sum\limits_{\sigma\in IBNC(\chi)} \left(\sum\limits_{\substack{ \pi\in IBNC(\chi)\\ \sigma\leq \pi\leq \epsilon}}\mu_{IBNC}(\sigma,\pi)\right)\phi_{\sigma} (z_1,\cdots ,z_n).$$
Therefore, the mixed moments of free-free-Boolean independent random variables  and the mixed moments of combinatorially free-free-Boolean independent random variables  are determined in the same way. The proof is done.
\end{proof}

\section{Bifree-Boolean central limit law}

In this section, we study an algebraic free-free-Boolean central limit theorem which is an analogy of Voiculescu's algebraic bi-free central limit theorem in \cite{Voi1}.   

Let $z=((z_i)_{i\in \I},(z_j)_{i\in \J},(z_k)_{k\in \K})$ be a three faced family of  random variables in $(\A,\phi)$. 
Notice that $IBNC(\chi)= P(2)$  when $\chi$ is map from $\{1,2\}$ to $\{\ell,r,c\}$ and the first order cumulant of a random variable is always the first moment of it.  
Therefore, the  second order free-free-Boolean cumulants are, as  free cumulants, variances and covariances of random variables.  
For convenience,  given $\omega:[n]\rightarrow\I\amalg \J\amalg \K $, we denote by $\chi_\omega:[n]\rightarrow \{\ell,r,c\}$ such that
$$\chi_\omega(i)=\left\{\begin{array}{ll}
\ell&\text{if}\quad\omega(i)\in \I,\\
r&\text{if}\quad\omega(i)\in \J,\\
c&\text{if}\quad\omega(i)\in \K.\\
\end{array}\right.
$$
Therefore, the  second order free-free-Boolean cumulants are as follows.
\begin{lemma}  \normalfont 
Let   $z=((z_i)_{i\in \I},(z_j)_{i\in \J},(z_k)_{k\in \K})$ be a three-faced family of  random variables in $(\A,\phi)$ and $\omega: \{1,2\}\rightarrow \I\amalg \J\amalg \K$.
Then
$$\kappa_{\chi_\omega,1_{[2]}}(z_{\omega(1)}z_{\omega(2)})= \phi(z_{\omega(1)}z_{\omega(2)}) -\phi(z_{\omega(1)})\phi(z_{\omega(2)}).$$
\end{lemma}

\begin{definition}\normalfont   A three-faced family of  random variables $z=((z_i)_{i\in \I},(z_j)_{i\in \J},(z_k)_{k\in \K})$ has a \emph{free-free-Boolean} central limit distribution if ,  for all $n\neq 2$,
$$\kappa_{\chi_\omega,1_{[n]}}(z_{\omega(1)},\cdots, z_{\omega(n)})= 0,$$
where $\omega:[n]\rightarrow  \I\amalg \J\amalg \K.$
\end{definition}

The following are  examples of free-free-Boolean families and   free-free-Boolean central limit distributions:

Let $\hh$ be a complex Hilbert space with orthonormal basis $\{e_i\}_{i\in \I}$ and  $\F(\hh)=\C\xi\oplus\bigoplus\limits_{n\geq 1}\hh^{\otimes n}$ be the full Fock space. 
 Let $\ell_i$ be the left creation operators on $\F(\hh)$ such that 	$\ell_i \xi=e_i$ and $\ell_i\zeta=e_i\otimes \zeta$ for all $\zeta\in \bigoplus\limits_{n\geq 1}\hh^{\otimes n}$.
  Let $r_i$ be the right creation operators on $\F(\hh)$ such that 	$r_i \xi=e_i$ and $r_i\zeta= \zeta\otimes e_i$ for all $\zeta\in \bigoplus\limits_{n\geq 1}\hh^{\otimes n}$. 
 Let $P_i$ be the orthogonal projection from $\F(\hh)$ onto $\C\xi\oplus \C e_i$. Then the  family of three-faced families of random variables
 $$\{((\ell_i,\ell_i^*),(r_i,r_i^*),(P_i\ell_iP_i,P_i\ell_i^*P_i))\}_{i\in \I}$$ 
is free-free-Boolean independent  in the probability space  $(B(\F(\hh)),\omega_\xi)$, where $(B(\F(\hh))$ is the set of all bounded operators on $\F(\hh)$ and $\omega_\xi=\langle\cdot\xi,\xi\rangle$ is the vacuum state on $(B(\F(\hh)) $. 
The space $\F(\C e_i)$, which is the Fock space generated by the one dimensional Hilbert space $\C e_i $, plays the role of $\X_i$ in Section 2.  
 
 Suppose that $\I$ has a disjoint partition that $\I=\bigcup\limits_{k\in \K} \I_k$. 
 For each $k$, let $\A_{k,\ell}$ be the unital $C^*$-algebra generated by $\{\ell_i|i\in I_k\}$,
 $\A_{k,r}$ be the unital $C^*$-algebra generated by $\{r_i|i\in I_k\}$  
 and $\A_{k,c}$ be the nonunital $C^*$-algebra generated by $\{P_k\ell_iP_k|i\in I_k\}$,
  where $P_k$ is the projection from $\F(\hh)$ onto the subspace generated by $\{\xi\} \cup \{e_i|i\in \I_k\}$. 
  Then the family of triples of faces  $\{(\A_{k,l},\A_{k,c},\A_{k,c})\}_{k\in \K}$ is free-free-Boolean in  $(B(\F(\hh)),\omega_\xi)$.

 	Moreover, one can easily obtain the following analogue of Theorem 7.4 in \cite{Voi1}.
\begin{proposition}\label{central limit}\normalfont 
There is exactly one free-free-Boolean central limit distribution $\Gamma_C:\C\langle Z_k| k\in  \I\amalg \J\amalg \K\rangle \rightarrow \C$ for a given matrix $C=(C_{k,l})$ with complex entries so that $\Gamma_C(Z_kZ_l)=C_{k,l}, k,l\in  \I\amalg \J\amalg \K$.

Let $h,h^*: \I\amalg \J\amalg \K\rightarrow \hh$ be  maps into the Hilbert space $\hh$ in the preceding example. 
Let 
$$ z_i=\ell(h(i))+\ell^*(h^*(i))$$ for $i\in \I$,
$$ z_j=r(h(j))+r^*(h^*(j))$$ for $j\in \J$
 and 
$$ z_k=P(\ell(h(k))+\ell^*(h^*(k)))P$$ for $k\in \K$, where $P$ is the orthogonal projection from $\F(\hh)$ onto $\C\xi\oplus\hh$ and $\ell(h(i))$ is the creation operator on $\F(\hh)$ such that $\ell(h(i))\xi=h(i)$ and $\ell(h(i))\zeta=h(i)\otimes \zeta$ for all $\zeta\in \bigoplus\limits_{n\geq 1}\hh^{\otimes n}$.
Then $z=((z_i)_{i\in \I},(z_j)_{i\in \J},(z_k)_{i\in \K})$ has a free,free-Boolean central limit distribution $\Gamma_C$ where $C_{k,l}=\langle h(l),h^*(k)\rangle$,$k,l\in \I\amalg \J\amalg \K$.
\end{proposition}

We end this section with an algebraic free-Boolean central limit theorem in analogue of Voiculescu's Theorem 7.9 in \cite{Voi1}

\begin{theorem}\normalfont
Let $z_n=((z_{n,i})_{i\in \I},(z_{n,j})_{i\in \J},(z_{n,k})_{i\in \K})$, $n\in \mathbb{N}$, be a free-free-Boolean sequence of three-faced families of random variables in $(\A,\phi)$, such that 
\begin{itemize}
\item[1.] $\phi(z_{n,l})=0$, for all $k\in \I\amalg \J\amalg \K$ and $n\in\mathbb{N}$.
\item[2.] $\sup_{n\in\mathbb{N}}|\phi(z_{n,l_1}\cdots z_{n,l_m})|=D_{l_1,\cdots,l_m}<\infty$ for every $l_1,\cdots,l_m\in\I\amalg \J\amalg \K.$
\item[3.] $\lim\limits_{N\rightarrow \infty} \frac{1}{N}\sum\limits_{n=1}^N\phi(z_{n,l}z_{n,l'})=C_{l,l'}$ for every $l',l\in \I\amalg \J\amalg \K$.
\end{itemize}
Let $S_N=((S_{N,i})_{i\in I},(S_{N,j})_{i\in J})$, where $S_{N,k}=\frac{1}{\sqrt{N}}\sum\limits_{n=1}^N z_{n,k}$ with $k\in \I\amalg \J\amalg \K.$  Let $\Gamma_C$ be the free-free-Boolean central limit distribution in Proposition \ref{central limit} with $C=(C_{k,l})$, $k,l\in \I\amalg \J\amalg \K$. Then, we have
$$\lim\limits_{N\rightarrow\infty} \mu_{S_N}(P)=\Gamma_C(P)$$
for all $P\in \C\langle Z_k|k\in I\cup J\rangle$.
\end{theorem}
\begin{proof}
Let $\omega:[m]\in \I\amalg \J\amalg \K$. By Remark 6.4 and the $n$-linearity of $\kappa_{\chi_\omega,1_{[m]}}$, we have
$$\kappa_{\chi_{\omega},1_{[m]}}(S_{N,\omega(1)},\cdots ,S_{N,\omega(m)})=\sum\limits_{n=1}^N {N^{-\frac{m}{2}}}\kappa_{\chi_{\omega},1_{[m]}}(z_{n,\omega(1)},\cdots ,z_{n,\omega(m)}).$$
When $m=1$, we have $\kappa_{\chi_{\omega},1_{[1]}}(z_{n,\omega(1)})=\phi(z_{n,\omega(1)})=0.$\\
When $m=2$, we have $\sum\limits_{n=1}^N {N^{-1}}\kappa_{\chi_{\omega},1_{[2]}}(z_{n,\omega(1)} z_{n\omega(2)})=\sum\limits_{n=1}^N {N^{-1}}\phi(z_{n,\omega(1)} z_{n,\omega(2)})=C_{\omega(1),\omega(2)}$.\\
When $m>2$,  since  $D_{k_1,\cdots,k_m}<\infty$ for all $k_1,\cdots,k_m$, $\sup\limits_{n}|\kappa_{\chi_{\omega},1_{[m]}}(z_{n,\omega(1)},\cdots, z_{n\omega(m)}) |<\infty.$ In this case, 
$$\lim\limits_{N\rightarrow \infty}\kappa_{\chi_{\omega},1_{[m]}}(S_{N,\omega(1)},\cdots, S_{N,\omega(m)})=0.$$
The proof is complete,  because moments are determined by polynomials of cumulants.
\end{proof}

\bibliographystyle{plain}

\bibliography{references}

\begin{thebibliography}{10}

\bibitem{CNS1}
Ian Charlesworth, Brent Nelson, and Paul Skoufranis.
\newblock Combinatorics of bi-freeness with amalgamation.
\newblock {\em Comm. Math. Phys.}, 338(2):801--847, 2015.

\bibitem{CNS}
Ian Charlesworth, Brent Nelson, and Paul Skoufranis.
\newblock On two-faced families of non-commutative random variables.
\newblock {\em Canad. J. Math.}, 67(6):1290--1325, 2015.

\bibitem{GHS}
Yinzheng Gu, Takahiro Hasebe, and Paul Skoufanis.
\newblock Bi-monotonic independence for pairs of algebras.
\newblock {\em available at https://arxiv.org/abs/1708.05334}, page 42 pages,
  2017.

\bibitem{GS}
Yinzheng Gu and Paul Skoufanis.
\newblock Bi-boolean independence for pairs of algebras.
\newblock {\em to appear in Complex Analysis and Operator Theory}, page 42
  pages, 2017.

\bibitem{GS1}
Yinzheng Gu and Paul Skoufranis.
\newblock Conditionally bi-free independence for pairs of faces.
\newblock {\em J. Funct. Anal.}, 273(5):1663--1733, 2017.

\bibitem{Liu3}
Weihua Liu.
\newblock Free-boolean independence for pairs of algebras.
\newblock https://arxiv.org/abs/1710.01374.

\bibitem{NS}
Alexandru Nica and Roland Speicher.
\newblock {\em Lectures on the combinatorics of free probability}, volume 335
  of {\em London Mathematical Society Lecture Note Series}.
\newblock Cambridge University Press, Cambridge, 2006.

\bibitem{Rota}
Gian-Carlo Rota.
\newblock On the foundations of combinatorial theory. {I}. {T}heory of
  {M}\"obius functions.
\newblock {\em Z. Wahrscheinlichkeitstheorie und Verw. Gebiete}, 2:340--368
  (1964), 1964.

\bibitem{Sp2}
Roland Speicher.
\newblock On universal products.
\newblock In {\em Free probability theory ({W}aterloo, {ON}, 1995)}, volume~12
  of {\em Fields Inst. Commun.}, pages 257--266. Amer. Math. Soc., Providence,
  RI, 1997.

\bibitem{SW}
Roland Speicher and Reza Woroudi.
\newblock Boolean convolution.
\newblock In {\em Free probability theory ({W}aterloo, {ON}, 1995)}, volume~12
  of {\em Fields Inst. Commun.}, pages 267--279. Amer. Math. Soc., Providence,
  RI, 1997.

\bibitem{VDN}
D.~V. Voiculescu, K.~J. Dykema, and A.~Nica.
\newblock {\em Free random variables}, volume~1 of {\em CRM Monograph Series}.
\newblock American Mathematical Society, Providence, RI, 1992.
\newblock A noncommutative probability approach to free products with
  applications to random matrices, operator algebras and harmonic analysis on
  free groups.

\bibitem{Voi1}
Dan-Virgil Voiculescu.
\newblock Free probability for pairs of faces {I}.
\newblock {\em Comm. Math. Phys.}, 332(3):955--980, 2014.

\end{thebibliography}

\noindent Department of Mathematics\\
Indiana University	\\
Bloomington, IN 47401, USA\\
E-MAIL: liuweih@indiana.edu \\

\end{document}